\documentclass[sn-mathphys,Numbered]{sn-jnl}

\usepackage{amsmath, amssymb,amsfonts}
\usepackage{amsthm}
\usepackage{mathrsfs}
\usepackage{xcolor}%
\usepackage{manyfoot}%

\usepackage{hyperref}
\usepackage{cleveref}
\usepackage{enumitem}

\usepackage{tikz}
\usetikzlibrary{matrix,arrows.meta}

\usepackage{amsmath,amssymb,amsfonts}%
\usepackage{amsthm}%
\usepackage{mathrsfs}%

\newcommand{\N}{\mathbb{N}}
\newcommand{\R}{\mathbb{R}}
\newcommand{\Z}{\mathbb{Z}}
\newcommand{\defeq}{\mathop{\mathrel{:}=}}
\renewcommand{\d}{\mathrm{d}}
\newcommand{\Hmm}[1]

\theoremstyle{plain}
\newtheorem{theorem}{Theorem}[section]
\newtheorem{lemma}[theorem]{Lemma}
\newtheorem{corollary}[theorem]{Corollary}

\theoremstyle{definition}
\newtheorem{remark}[theorem]{Remark}
\newtheorem{mydef}[theorem]{Definition}

\begin{document}

\title[Article Title]{Positive Harmonic Functions on Graphs with Nilpotent Group Actions}

\author[1]{\fnm{Matti} \sur{Richter}}\email{matti.richter@uni-potsdam.de}

\affil[1]{\orgdiv{Institute of Mathematics}, \orgname{University of Potsdam}, \orgaddress{\street{Am Neuen Palais 10}, \city{Potsdam}, \postcode{
14469}, \state{Brandenburg}, \country{Germany}}}


\abstract{
    We study directed weighted graphs which are invariant under a nilpotent and cocompact group action. In particular, we consider the conic section $\mathcal{K}$ of the set of positive harmonic functions. We characterise the set of extreme points of the convex and compact set $\mathcal{K}$ as the set of multiplicative elements in $\mathcal{K}$. Moreover, we study positive generalised eigenfunctions for a given parameter $\lambda$. We find that the topological space $\mathcal{M}_\lambda$ of multiplicative $\lambda$-harmonic functions is homeomorphic to a sphere for $\lambda$ below a certain threshold.}

\maketitle

\textbf{Acknowledgements.}
The author thanks Yehuda Pinchover for a valuable conversation which inspired the paper. Moreover, Matthias Keller is acknowledged for his continuous feedback regarding both content and structure of the work.

\section*{Preface}
Based on a seminal paper \cite{choquetDeny} by Choquet and Deny,  
Margulis studied positive harmonic functions on discrete nilpotent groups \cite{margulis}. Using Choquet's theorem, he proved that the harmonic functions can be represented by an integral over the multiplicative ones. In this setting, multiplicative functions are group homomorphisms from the group to the positive real numbers. An immediate consequence of this result is a Liouville property, i.e. any bounded harmonic function is constant. This propety had also been shown earlier for Markov chains with nilpotent transitive group actions by Dynkin and Maljutov in \cite{dynkin}. 

Later, Lin \cite{lin88} employed ideas of Margulis' proof to a continuous setting of linear differential operators. He showed that for any cocompact group action all bounded positive harmonic functions which satisfy $f(ghx) = f(hgx)$ for all group elements $g,h$ are constant.

Moreover, Agmon \cite{agmon} studied second order linear elliptic operators on $\R^d$ with real periodic coefficients and showed results very similar to Margulis. Inspired by the results of Allegretto \cite{allegretto1} \cite{allegretto2} \cite{allegretto3} and Piepenbrink  \cite{piepenbrink2} \cite{piepenbrink3} 
\cite{piepenbrink1} he considered positive harmonic functions for such operators. Using a proof of Pinchover, he proved an integral representation over harmonic exponentials, which are functions of the form $x \mapsto \exp(\langle \xi, x \rangle) \varphi(x)$ where $\xi \in \R^d$ and $\varphi$ is periodic. Agmon further investigated the structure of the family of sets of $\lambda$-harmonic exponentials for a real parameter $\lambda$. This family was also studied on symmetric spaces by Guivarc'h in \cite{guivarch}. Then in 1994, Lin and Pinchover \cite{linPinchover} generalised the results of Agmon to manifolds with nilpotent and cocompact group actions. In their setting, the operator is an elliptic operator which is invariant under the group action. Then the positive harmonic functions are represented via integrals over multiplicative harmonic functions. Lin and Pinchover also made use of Margulis' arguments regarding nilpotent group actions, which are not needed for the abelian case of $\Z^d$ acting on $\R^d$. In this sense Agmon's result is a special case of \cite{linPinchover}.

The setting of this paper is an directed, possibly non locally finite, graph with a nilpotent and cocompact group action. We study the positive generalised eigenfunctions of a Schrödinger operator in two steps.

First, we generalise the representation theorem of Margulis \cite{margulis}. Second, we describe the structure of the topological spaces of positive multiplicative generalised eigenfunctions in the spirit of \cite{agmon} and \cite{linPinchover}. 

The paper is structured as follows.
The first section introduces the relevant terminology and notation. 
The second section is devoted to the first main result, Margulis' theorem for group actions, \Cref{thm1}.
Here the convex and compact set $\mathcal{K}$ is studied which is a closure of the set of normalized positive harmonic functions. The theorem states that every harmonic element of $\mathcal{K}$ is invariant under the commutator subgroup and that the harmonic extreme points are multiplicative. As a consequence of this, we obtain a representation theorem for positive harmonic functions in \Cref{CorChoq}. As a further consequence, this reduces the study of positive harmonic functions to the case of an abelian group, see Section 2.2.

The third section is concerned with analogues of the results of Agmon and Lin and Pinchover. We assume that the graph is locally finite, to show additional results. The second main result, \Cref{MKisexK} characterises the extreme points of $\mathcal{K}$ as the set of multiplicative elements of $\mathcal{K}$. One inclusion follows directly from \Cref{thm1}, and the other is shown here with some additional work. After that, we extend our studies to positive $\lambda$-harmonic functions for a parameter $\lambda \in \R$. Since these are just harmonic functions for a different operator, we can apply \Cref{MKisexK} for each parameter. In view of this, we then go on to study the family of sets $\mathcal{M}_\lambda$ of positive, multiplicative $\lambda$-harmonic functions. The third result, \Cref{thm3} then describes the topological properties of this family. More precisely, it states that $\mathcal{M}_\lambda$ is either empty, a one-point set or homeomorphic to a $d$-sphere, depending on the parameter  $\lambda$. Here $d$ is an integer which is independent of $\lambda$.

\section{Setting}\label{setting}

\subsection{Graphs with group actions and a Harnack inequality}

We begin with some basic definitions regarding graphs and Schrödinger operators.

Let $X$ be a countable discrete space. Then a \emph{graph} over $X$ is defined as a weight function $b : X \times X \to [0, \infty)$ together with a potential $c : X \to \R$ such that 
$$ \sum_{y \in X} b(x,y) < \infty \quad$$
holds for all $x \in X$. We also usually assume $b(x,x) = 0$ for all $x \in X$ but this does not affect the harmonic functions. Moreover, some graphs with nonzero diagonal terms will occur later. Next, we define the \emph{degree} of a vertex $x \in X$ by
$$\deg(x) \defeq \sum_{y \in X} b(x,y) + c(x).$$
Note that even though this notation is in the style of \cite{keller} we do not assume that $b$ be symmetric here.

We say $y$ is a \emph{neighbour} of $x$ if $b(x,y) > 0$, and write $x \sim y$. We say the graph is \emph{locally finite} if every $x \in X$ has only finitely many neighbours, and we say the graph is connected, if for every $x,z \in X$ there is a path $x = y_1 \sim \dots \sim y_n = z$ in $X$. 

For a given graph $(b,c)$, we then define the \emph{Schrödinger operator} $H = H_{b,c}$ by

$$H_{b,c} : \mathrm{Dom}(H_{b,c}) \to C(X), \quad H_{b,c}f(x) \defeq \sum_{y \in X} b(x,y)(f(x) - f(y)) +c(x)f(x),$$
where $C(X)$ denotes the space of all (continuous) real-valued functions on $X$ and

$$\mathrm{Dom}(H_{b,c}) \defeq \{f \in C(X) \mid \sum_{y \in X}b(x,y)|f(x)| < \infty \}.$$
In the literature, the operator $H$ is also called \emph{Laplacian} when $c = 0$ or $c \geq 0$. However, we want to allow $c$ to assume arbitrary real values. We define the subspace of \emph{harmonic functions} as
$$\mathcal{H} \defeq \{ f \in \mathrm{Dom}(H) \mid Hf = 0\}$$
and the convex cone of \emph{superharmonic functions} as
$$\mathcal{S} \defeq \{ f \in \mathrm{Dom}(H) \mid Hf \geq 0\}.$$

Moreover, a function is called \emph{subharmonic}, if $(-f) \in \mathcal{S}$. We also write $\mathcal{H}^+$ and $\mathcal{S}^+$ for the set of nontrivial and non-negative elements of $\mathcal{H}$ and $\mathcal{S}$ respectively. It is easy to see that $\mathcal{H}^+$ and $\mathcal{S}^+$ are convex cones, and we will later see that they contain only strictly positive elements.
Note that a function $f \in \mathrm{Dom}(H)$ is superharmonic if and only if $Hf \geq 0$, which is equivalent to
$$\sum_{y \in X} b(x,y)f(y) \leq \deg(x)f(x), \quad \text{ for all } x \in X.$$
This can be seen by simply rearranging terms, and equality holds for all $x$ if and only if $f$ is harmonic. We will make use of this characterisation later.

Let $G$ be a (discrete) group acting on $X$. Then the group action is called \emph{cocompact} if there exists a compact (i.e. finite) $V \subseteq X$ such that $GV = X$. We call such a set $V$ a \emph{fundamental domain} if the orbits of elements in $V$ are pairwise disjoint. If the group action is cocompact, then clearly a fundamental domain $V$ always exists, as for a finite $V$ with more elements one may simply remove some.

In the following we will fix the space $X$ and a cocompact group action $G$ on $X$ with fundamental domain $V$. For $g \in G$, define the \emph{shift} by $g$ as
$$T_g : C(X) \to C(X), \quad T_gf(x) = f(g^{-1}x),$$
which is linear and continuous with respect to the product topology of $\R$ over $X$. Given a subgroup $R \subseteq G$, we say a function $f$ is $R$-\emph{multiplicative} with the \emph{character} $\gamma \in \operatorname{Hom}(R,\R_{>0})$ if 
$$T_gf = \gamma(g^{-1})f \quad \text{ for  all $g \in R$},$$ 
where $\operatorname{Hom}(R,\R_{>0})$ denotes the set of group homomorphisms from $R$ to the group of positive real numbers under multiplication. We only consider homomorphisms into the positive real numbers, as we will mainly be interested in positive functions. In the case of $R = G$ we will often just say that $f$ is \emph{multiplicative}. The corresponding character to such an $f$ will be denoted by $\gamma_f$. In the case of $\gamma_f = 1$, i.e. $T_gf = f$ for all $g \in R$, we say $f$ is $R$-invariant.

We say an operator $H$ on $C(X)$ is $R$-invariant, if $T_g(\operatorname{Dom}(H)) \subseteq \operatorname{Dom}(H)$ and
$$H \circ T_g = T_g \circ H\quad \text{ for all $g \in R$},$$
i.e. the operator commutes with the shift. Then, if a function $f$ is superharmonic with respect to an $R$-invariant Schrödinger operator, it follows that $T_gf$ is also superharmonic for all $g \in R$ as
$$H(T_gf) = T_g(Hf) \geq 0,$$
and analogous statements hold for subharmonicity and harmonicity. It turns out that the notion of $R$-invariance of the operator is equivalent to a similarly natural definition of $R$-invariance of the graph $(b,c)$.
\begin{lemma}\label{HGinv}
Let $R \subseteq G$ be a subgroup and $(b,c)$ be a graph over $X$. Then the Schrödinger operator $H = H_{b,c}$ is $R$-invariant if and only if $b$ and $c$ are both $R$-invariant, i.e. if
$$ b(gx,gy) = b(x,y) \text{ and } c(gx) = c(x) \text{ for all $x,y \in X$ and $g \in R$}.$$
Moreover, in this case $\deg(gx) = \deg(x)$ for all $x \in X$ and $g \in R$.
\begin{proof} Let $b, c$ be $R$-invariant, then for $f \in \operatorname{Dom}(H)$ and $x \in X$ we have
\begin{align*}
    HT_{g^{-1}}f(x) &= \sum_{y \in X} b(x,y)(f(gx) - f(gy)) + c(x)f(gx) \\
    &= \sum_{y \in X} b(gx,gy)(f(gx) - f(gy)) + c(gx)f(gx) \\
    &= \sum_{y \in X} b(gx, y)(f(gx) - f(y)) + c(gx)f(gx) \\
    &= Hf(gx) = T_{g^{-1}}Hf(x)
\end{align*}
and clearly $T_{g^{-1}}f \in \operatorname{Dom}(H)$ for all $g \in R$.

For the converse, assume that $H$ is $R$-invariant and let $g\in R$. Then for all $x \in X$ we get
$$c(x) = H1(x) = HT_{g^{-1}}1(x) = T_{g^{-1}}H1(x) = c(gx)$$
where $1$ denotes the constant function equal to $1$ on $X$.
Moreover, for $x \neq y \in X$ we have $gx \neq gy$ and
\begin{align*}b(x,y) &= -\left(\sum_{z \in X}b(x,z)(1_{gy}(gx) - 1_{gy}(gz)) +c(x)1_{gy}(gx)\right) \\ &= -HT_{g^{-1}}1_{gy}(x) = -T_{g^{-1}}H1_{gy}(x) \\
&= -\left(\sum_{z \in X}b(gx,z)(1_{gy}(gx) - 1_{gy}(z)) + c(gx)1_{gy}(gx)\right) \\
&= b(gx,gy), \end{align*}
and $b(x,x) = 0 = b(gx, gx)$ for all $x \in X$.
For the 'moreover' part, we simply observe that for $x \in X$ and $g \in R$ we have
$$\deg(gx) = \sum_{y \in X} b(gx,y) + c(gx) = \sum_{y \in X} b(x,g^{-1}y) + c(x) = \deg(x),$$
as the action $y \mapsto g^{-1}y$ is bijective.
\end{proof}
\end{lemma}

With this at hand, we can prove a Harnack inequality which is uniform with respect to the group action.

\begin{lemma}[Harnack inequality]\label{harnack}
Let $b$ be a connected graph and let $c$ be a potential such that $H = H_{b,c}$ is $R$-invariant. Then, for $x, y \in X$, there exists $C_{x,y} > 0$ such that
$$T_gf(x) \geq C_{x,y}T_gf(y)$$
holds for every $f \in \mathcal{S}^+$ and every $g \in R$.
\end{lemma}
\begin{proof} If $X$ consists only of a single point or $\mathcal{S}^+ = \emptyset$, the statement is trivial. Assume $\#X \geq 2$ and let $f \in \mathcal{S}^+$. If $\deg(x) \leq 0$ for some $x \in X$, we have
$$0 \leq Hf(x) = \deg(x)f(x) - \sum_{y \in X}b(x,y)f(y) \leq 0$$
and hence $f(y) = 0$ follows whenever $b(x,y)$ does not vanish. Moreover, in this case we also have $\deg(y)f(y) = 0 \leq 0$, and it follows for all $x \sim y \in X$ that if $z \mapsto f(z)\deg(z)$ is not positive at $x$ then it is also not positive at $y$ and $f(y) = 0$. Since the graph is connected, this argument yields that $z \mapsto f(z) \mathrm{deg}(z)$ vanishes everywhere and $f = 0$, which is a contradiction to $f \in \mathcal{S}^+$. It follows that $\deg$ is strictly positive, so superhamonicity of $f$ in $gx$ yields
$$f(gx) \geq \deg(gx)^{-1}\sum_{gy \in X}b(gx,gy)f(gy) \geq \deg(x)^{-1}b(x,y)f(gy)$$
for all $g \in G$ and $x \sim y \in X$, using $b(gx,gy) = b(x,y)$ and $\deg(gx) = \deg(x)$. Finally, the statement follows by an inductive argument using connectedness of the graph, as the constants $ \mathrm{deg}(x)^{-1} b(x,y)$ are independent of $f$ and $g$.
\end{proof}

\begin{corollary}\label{harCor1}
For a Schrödinger operator $H$ associated to a connected graph every $f \in \mathcal{S}^+$ is strictly positive. 
\begin{proof}
First, note that the subgroup $R = \{ e\}$ acts trivially on $X$ such that $H$ is $R$-invariant.
Let $f \in \mathcal{S}^+$, in particular $f(x) > 0$ for some $x \in X$. It follows immediately from the Harnack inequality that $f(y) \geq C_{y,x}f(x) > 0$ for every $y \in X$, which shows $f > 0$.
\end{proof}
\end{corollary}

From this we see that $\mathcal{H}^+$ contains only strictly positive elements. Moreover, we get the following compactness result for a section of the cone $\mathcal{S}^+$.
\begin{corollary}\label{Scomp}
Let $H$ be a Schrödinger operator associated to a connected graph $(b,c)$ on $X$ and fix $x_0 \in X$. Then the conic section  
$$\mathcal{S}^+_1 \defeq \{f \in \mathcal{S}^+\mid f(x_0) = 1 \}$$ is compact with respect to the product topology of $\R$ over $X$.

\begin{proof}
As in the last corollary, we apply the Harnack inequality for the trivial subgroup $\{e\}$. We first show that $\mathcal{S}_1^+$ is closed. Let $f_n \in C(X)$ be a sequence converging to $f \in C(X)$ in the product topology, i.e. it converges pointwise. Then $f$ is clearly still non-negative and $f(x_0) = 1$ so $f$ is also nontrivial, and by Fatou's lemma we have
$$\sum_{y \in X} b(x,y)f(y) \leq \liminf_{n \to \infty} \sum_{y \in X} b(x,y) f_n(y) \leq \deg(x) \lim_{n \to \infty}f_n(x) = \deg(x)f(x),$$
so $f$ is superharmonic, hence $f \in \mathcal{S}^+_1$ and $\mathcal{S}^+_1$ is closed.

Moreover, for any $f \in \mathcal{S}^+_1$ it follows from the Harnack inequality and $f(x_0) = 1$ that $f(x) \in [C_{x,x_0}, C_{x_0,x}^{-1}]$, and hence

$$\mathcal{S}^+_1 \subseteq \prod_{x \in X} [C_{x,x_0}, C_{x_0,x}^{-1}]$$
where the right hand side is compact by Tychonoff's theorem, and compactness of $\mathcal{S}^+_1$ follows readily.
\end{proof}
\end{corollary}

It should be noted that the set $\mathcal{H}^+_1$ is not necessarily closed, so we cannot conclude its compactness. This issue is addressed in the following section.

\subsection{Convex and compact sets and Choquet theory}
Let $(b,c)$ be a connected graph on $X$ with a given $G$-action such that $H = H_{b,c}$ is $G$-invariant. Moreover, fix a point $x_0 \in X$.
On $C(X)$, we will consider the product topology of $\R$ over $X$. For a set $A \subseteq C(X)$, we denote by $\mathrm{ex} \: A$ the set of extreme points and by $\overline{\mathrm{conv}} A$ the closed convex hull. We study the convex cone $\mathcal{H}^+$ of positive harmonic functions. However, in order to apply Choquet theory, we need compact and convex sets, so we only take a section of this cone and then take the closure, leading to the definition
$$\mathcal{K} \defeq \overline{\{f \in \mathcal{H}^+ \mid f(x_0) = 1\}}.$$
 Taking only those functions with $f(x_0) = 1$ is no real loss of generality, since the Schrödinger operator is linear, so scaling a function by a positive number commutes with the operator $H$. On the other hand, taking the closure slightly complicates the situation for non-locally-finite graphs, since the functions in $\mathcal{K}$ may not all be harmonic. However, we still have $\mathcal{K} \subseteq \mathcal{S}^+_1$ since $\mathcal{H}^+ \subseteq \mathcal{S}^+$  and $\mathcal{S}^+_1$ is compact by \Cref{Scomp}, which also shows that $\mathcal{K}$ is compact. Moreover, it follows immediately from the definition that $\mathcal{K}$ is convex.

For a subgroup $R \subseteq G$ we define
$$\mathcal{K}^R \defeq \{f \in \mathcal{K} \mid f \text{ is $R$-invariant}\} = \mathcal{K} \cap \bigcap_{g \in R} \mathrm{ker}(T_g - I),$$
where $I$ denotes the identity on $C(X)$ and $\mathrm{ker}(L)$ denotes the kernel of an operator $L$. Since the kernel of $T_g -I$ is a closed subspace of $C(X)$ for every $g \in G$, and $\mathcal{K}$ is compact and convex, it follows that $\mathcal{K}^R$ is compact and convex as well. 

The sets $\mathcal{K}^R$ are defined this way so we may apply Choquet's theorem, see \cite{phelps}, which states the following: Let $K$ be a compact and convex subset of a locally convex vector space such that $K$ is metrizable. Then, for each $f \in K$ there exists a Borel measure $\mu_f$ on the set of extreme points $\mathrm{ex} \: K$ such that

$$f = \int\displaylimits_{\mathrm{ex}\:K} k \:d\mu_f(k).$$

This turns out to be particularly useful for finding harmonic functions. Indeed, taking $K = \mathcal{K}$ we will see that almost all elements of $\mathrm{ex} \: \mathcal{K}$ are harmonic and multiplicative, when the measure $\mu_f$ is induced by a harmonic element of $\mathcal{K}$. Then the problem of finding all positive harmonic functions reduces to just finding the multiplicative ones. The first part of this statement is shown in the following lemma, due to which we only need to study the harmonic extreme points of $\mathcal{K}$. A version of this lemma for the special case of Cayley Graphs can be found in \cite[Equation (3)]{margulis}.

\begin{lemma}\label{harmEx}
Let $R \subseteq G$ and $f \in \mathcal{K}^R$ be harmonic. Let $\mu_f$ be the measure obtained by Choquet's theorem such that
$$f = \int\displaylimits_{\mathrm{ex}\: \mathcal{K}^R} k \: d\mu_f(k),$$
then $\mu_f$ is concentrated on the harmonic functions, i.e. the set
$$\{k \in \mathrm{ex} \: \mathcal{K}^R \mid Hk \neq 0\}$$
is a null set with respect to $\mu_f$.

\begin{proof}
For a given $x \in X$, harmonicity of $f$ yields $\deg(x)f(x) = \sum_{y \in X}b(x,y)f(y)$, and hence
\begin{align*}
    \int\displaylimits_{\mathrm{ex}\: \mathcal{K}^R} \deg(x)k(x) \: d\mu_f(k) &= \deg(x)\int\displaylimits_{\mathrm{ex}\: \mathcal{K}^R} k(x) \: d\mu_f(k) \\
    &= \deg(x)f(x) \\
    &= \sum_{y \in X}b(x,y)f(y) \\
    &= \sum_{y \in X}b(x,y)\int\displaylimits_{\mathrm{ex}\: \mathcal{K}^R} k(y) \: d\mu_f(k) \\
    &= \int\displaylimits_{\mathrm{ex}\: \mathcal{K}^R} \sum_{y \in X}b(x,y)k(y) \: d\mu_f(k),
\end{align*}
where we are allowed to interchange the sum and the integral since $k$ and $b$ are non-negative. Since the left hand side is clearly finite, it follows that
$$ \int\displaylimits_{\mathrm{ex}\: \mathcal{K}^R} Hk(x) \: d\mu_f(k) = 0.$$
Moreover, since every $k \in \mathcal{K}^R$ is superharmonic, the integrand is nonnegative. From this it follows that for each $x \in X$ the set of extreme points with $Hk(x) \neq 0$ has measure zero and, as $X$ is countable, we conclude that 
$$\{k \in \mathrm{ex}\:\mathcal{K}^R \mid Hk \neq 0\} = \bigcup_{x \in X} \{k \in \mathrm{ex} \:\mathcal{K}^R \mid Hk(x) \neq 0\}$$
is a null set with respect to $\mu_f$.
\end{proof}
\end{lemma}

\section{Positive harmonic functions for nilpotent groups}

\subsection{Invariance under the commutator and multiplicativity}

In this chapter, we focus on a group action by a nilpotent group on a graph. In particular, we want to study the set $\mathcal{K}$ introduced in the previous chapter in more detail. We prove a theorem, which generalises the results of \cite{margulis} from Cayley Graphs on a nilpotent group to cocompact graphs with a nilpotent group action. To this end, fix a connected graph $(b,c)$ over $X$ with nilpotent and cocompact group action $G$ and fundamental cell $V$ and fix $x_0 \in X$. Recall that $\mathcal{K} = \overline{\{f \in \mathcal{H}^+ \mid f(x_0) =1\}}$. The first main result of the paper is the following theorem.

\begin{theorem}[Margulis' theorem for group actions]\label{thm1}
Let $(b,c)$ be a connected graph on $X$ with a nilpotent and cocompact group action $G$ such that $H_{b, c}$ is $G$-invariant. Fix $x_0 \in X$, then the set $\mathcal{K}$ satisfies the following:
\begin{enumerate}[label=(\alph*)]
    \item Every element of $\mathcal{K}$ is $[G,G]$-invariant.
    \item Every harmonic extreme point of the convex set $\mathcal{K}$ is multiplicative.
\end{enumerate}
\end{theorem}

\begin{corollary}\label{CorChoq}
In the setting of \Cref{thm1}, for every $f \in \mathcal{H}^+$ there exists a probability measure $\mu_f$ supported on the harmonic and multiplicative elements of $\mathcal{K}$ such that
$$f = f(x_0)\int\displaylimits_{\mathcal{K}} k d \mu_f.$$

\begin{proof}
We observe that $ \tilde{f} := f/f(x_0) \in\mathcal{K}$. Let $\mu_f$ be the representing measure of $\tilde{f}$ obtained by Choquet's theorem. Then the equation holds and it follows from \Cref{harmEx} that $\mu_f$ is concentrated on the harmonic extreme points of $\mathcal{K}$. Finally, \Cref{thm1} yields that these elements are multiplicative.
\end{proof}
\end{corollary}

In order to prove \Cref{thm1} we first make some simple group theoretic observations.
\begin{mydef}
For a normal subgroup $R \subseteq G$, consider the natural projection $\pi : G \to G/R$ and the centre $Z(G/R)$ of the factor group. We define
$$Q(R) \defeq \pi^{-1}(Z(G/R)).$$
\end{mydef}

We note that $R \subseteq Q(R)$ since $R = \ker \pi$, and $Q(R)$ is a normal subgroup of $G$ as $Z(G/R)$ is normal and $\pi$ is a group homomorphism. Moreover, we obtain the following lemma.
\begin{lemma}\label{Qcomm}
Let $f \in \mathcal{K}^R$, $q \in Q(R)$ and $g \in G$. Then 
$$T_q T_gf = T_g T_qf.$$
\begin{proof}
Since $\pi(q) \in Z(G/R)$ we have $\pi([q,g]) = \pi(q)^{-1} \pi(g)^{-1} \pi(q) \pi(g) =  e_{G /R}$ and hence $[q,g] \in R$, so the identity
$$ T_{qg}f= T_{gq}T_{[q,g]}f = T_{gq}f$$
follows from $R$-invariance of $f$.
\end{proof}
\end{lemma}

We now start the main part of the proof of \Cref{thm1}, for which we need two key lemmas. The first, \Cref{lem1}, states, that the harmonic extreme points
of $\mathcal{K}^R$ are $Q(R)$-multiplicative for any normal subgroup $R$. Then the second one, \Cref{lem2}, uses the first lemma and states that the harmonic extreme points are even invariant under elements $q \in Q(R)$ which are commutators. These two lemmas are analogous to \cite[Lemma 1 and Lemma 2]{margulis} and  \cite[Lemma 6.3 (i) and (ii)]{linPinchover}.

With these at hand, the idea of the proof is the following: We know that any harmonic $f \in \mathrm{ex} \: \mathcal{K}^R$ is also invariant under commutators in $Q(R)$. Then Choquet's theorem together with \Cref{harmEx} yields that this is true for all $f \in \mathcal{K}^R$, not just for extreme points. This shows $\mathcal{K}^R \subseteq \mathcal{K}^{R'}$, where $R'$ is the subgroup generated by commutators in $Q(R)$. Taking $R = G_{n+1}$, an element of the lower central series of $G$, we then find that $G_n \subseteq R'$. By an inductive argument, it follows that $\mathcal{K}^{G_n} \subseteq \mathcal{K}^{G_1}$ for all $n \in \N$. Finally, nilpotency of the group states that there exists an $n \in \N$ such that $G_n = \{e\}$. Combining this with the fact that $G_1 = [G,G]$ we conclude that
$$\mathcal{K} = \mathcal{K}^{[G,G]}.$$
Applying \Cref{lem1} one more time, we also obtain that the harmonic $f \in \mathrm{ex}\:\mathcal{K}$ are multiplicative with respect to $Q([G,G]) = G$.

\begin{lemma}\label{lem1}
Let $R$ be a normal subgroup of $G$. Then every harmonic $k \in \mathrm{ex} \: \mathcal{K}^R$ is $Q(R)$-multiplicative.

\begin{proof}
Let $k \in \mathrm{ex} \: \mathcal{K}^R$ be harmonic and $q \in Q(R)$. For all $r \in R$, we have
$$T_rT_qk = T_qT_rk = T_qk$$
by \Cref{Qcomm}, so $T_qk$ is $R$-invariant. Since the group action is cocompact, choose a fundamental domain $V \subseteq X$ such that $GV = X$. Then, using the Harnack inequality, we find $C > 0$ such that $T_gk(v) > CT_gk(q^{-1}v)$ for all $v \in V$ and $g \in G$. Then, together with \Cref{Qcomm} we obtain
\begin{align*}
    k(g^{-1}v) &= T_{g}k(v) > CT_{g}k(q^{-1}v) \\
        &= CT_{g}T_qk(v) = CT_qT_gk(v) = CT_qk(g^{-1}v)
\end{align*} for all $v \in V$ and $g \in G$, which implies $k > CT_qk$ as $GV = X$.

From this it follows that the auxilliary function $s \defeq k - CT_qk$ is strictly positive. Moreover, $s$ is harmonic and $R$-invariant as both $k$ and $T_qk$ are. Then the normalised functions $k(q^{-1}x_0)^{-1}T_qk$ and $s(x_0)^{-1}s$ are elements of $\mathcal{K}^R$ and
$$k = Ck(q^{-1}x_0)\frac{T_qk}{k(q^{-1}x_0)} + s(x_0)\frac{s}{s(x_0)}$$
is a convex combination, as $Ck(q^{-1}x_0) + s(x_0) = k(x_0) = 1$. Since $k$ is an extreme point it follows that $k = k(q^{-1}x_0)^{-1}T_qk$ and thus
$$T_qk = k(q^{-1}x_0)k,$$
i.e. $k$ is $Q(R)$-multiplicative with $\gamma_k(q) = k(qx_0)$.
\end{proof}
\end{lemma}

\begin{lemma}\label{lem2}
Let $R$ be a normal subgroup of $G$, then for every $q \in Q(R)$ which is a commutator in $G$ we have 
$$T_qf = f$$
for each harmonic $f \in \mathcal{K}^R$.

\begin{proof}
Let $q \in Q(R)$ such that $q = [g,l]$ with $g,l \in G$.
We first show the statement for the extreme points, so let $k \in \mathrm{ex}\: \mathcal{K}^R$. Since we know that $\pi([g,l]) = \pi(q) \in Z(G/R)$, we get
$$\pi(q^n) = \pi([g,l]^n) = [\pi(g), \pi(l)]^n = [\pi(g), \pi(l)^n] = \pi([g,l^n])$$
and hence
$$[g,l^n] = q^nr_n$$
with $r_n \in R$ for  every $n \in \N$.
The function $k$ is $Q(R)$-multiplicative with character $\gamma_k$ by \Cref{lem1} and from \Cref{Qcomm} it follows that
$$T_qT_g^nk = T_g^nT_qk = T_g^n\gamma_k(q^{-1})k$$
for each $n \in \N$.
Therefore, for every element $h$ of the set 
$$\mathcal{C}_0 = \{T_g^nk \mid n \in \N_0\},$$
 it follows that $h$ is $Q(R)$-multiplicative with the same character $\gamma_k$. Moreover, every $h \in \mathcal{C}_0$ is $R$-invariant, as
$$T_rh = T_rT_g^nk = T_g^nT_{g^{-n}rg^n}k = T_g^nk = h$$
for all $r \in R$ since $R$ is normal and $k$ is $R$-invariant. Indeed, we have $Q(R)$-multiplicativity and $R$-invariance for each element of the closed cone $\mathcal{C}$ generated by $\mathcal{C}_0$, since these properties are retained by positive linear combinations and limits. Setting $\mathcal{C}_1 \defeq \{h \in \mathcal{C} \mid h(x_0) = 1\}$ we can see that $\mathcal{C}_1$ is convex and compact, as $\mathcal{C}_1 \subseteq \mathcal{K}^R$. Now the continuous endomorphism
$$\mathcal{C}_1 \to \mathcal{C}_1, \quad h \mapsto h(g^{-1}x_0)^{-1}T_gh$$
must have a fixed point $h_1 \in \mathcal{C}_1$ by the Schauder-Tychonoff fixed point theorem \cite{schauderTych}, and for all $n \in \N$ we obtain 
\begin{align*}
    T_gT_l^nh_1 &= T_l^nT_g T_{[g, l^n]}h_1 \\
                &= T_l^n T_g T_q^n T_{r_n}h_1 \\
                &= T_l^n T_g \gamma(q^{-1})^n h_1 \\
                &= \gamma(q)^{-n} h_1(g^{-1}x_0) T_l^n h_1,
\end{align*}
using $[g,l^n] = q^nr_n$ and the fact that $h_1$ is $R$-invariant and $Q(R)$-multiplicative.
Evaluating this identity at $x_0$ and rearranging yields
$$ \frac{T_l^nh_1(g^{-1}x_0)}{T_l^nh_1(x_0)} = \gamma(q)^{-n}h_1(g^{-1}x_0)$$
for all $n \in \N$. In this equation, the left hand side and its inverse are bounded by the Harnack inequality, so it follows that $\gamma(q) = k(qx_0) = 1$ and hence $T_qk = k$.

Then, by Choquet's theorem and \Cref{harmEx}, we can extend the $q$-invariance to arbitrary $f \in \mathcal{K}^R$, as
$$f(qx) = \int\displaylimits_{k \in \mathrm{ex}\:\mathcal{K}} k(qx) \d\mu_f = \int\displaylimits_{k \in \mathrm{ex}\:\mathcal{K}} k(x) \d\mu_f = f(x).$$
This finishes the proof.
\end{proof}
\end{lemma}

We are now ready to prove \Cref{thm1} as outlined above.

\begin{proof}[Proof of \Cref{thm1}]
Consider the lower central series of $G$, denoted by $(G_n)_{n \in \N}$. For a fixed $n \in \N$, taking the natural projection $\pi :G \to G/G_{n+1}$ we have
$$G_n \subseteq \pi^{-1}(G_n/G_{n+1}) \subseteq \pi^{-1}(Z(G/G_{n+1})) = Q(G_{n+1}),$$
as $G_n/G_{n+1} \subseteq Z(G/G_{n+1})$. Moreover, we know that $G_n = [G, G_{n-1}]$ is generated by commutators which are also in $Q(G_{n+1})$. It follows by \Cref{lem2} that any harmonic and $G_{n+1}$-invariant $f \in \mathcal{K}$ is also $G_n$-invariant. Iterating this argument, we find that any such $f$ is invariant under $G_1 = [G,G]$. By nilpotency of $G$, we find $n \in \N$ such that $G_n = \{ e \}$, so every $f \in \mathcal{K}$ is $G_n$-invariant, and therefore $[G,G]$-invariant. Since the harmonic elements are dense in $\mathcal{K}$, we find
$$\mathcal{K} = \mathcal{K}^{[G,G]},$$
showing (a). Moreover, applying \Cref{lem1} one more time with $R = [G,G]$, we obtain that every harmonic extreme point of $\mathcal{K} = \mathcal{K}^{[G,G]}$ is $Q([G,G])$-multiplicative. Since $G/[G,G]$ is abelian it follows that $Q([G,G]) = G$, and therefore every harmonic extreme point of $\mathcal{K}$ is $G$-multiplicative, showing (b). 
\end{proof}

\subsection{Factoring by a normal subgroup}
We have seen in the previous subsection that every positive harmonic function is invariant under the commutator subgroup. Since factoring a group by its commutator subgroup yields an abelian group, it would be useful to have a statement which allows us to assume that $G$ is abelian in this setting without loss of generality. 
In this chapter we make this idea precise: We define the notion of a factored space $X_R$ and graph $(b_R, c_R)$ on which $G / R$ acts for any normal subgroup $R$. Crucially, this factoring retains many relevant properties of the graph, in particular the set of positive harmonic functions, provided every such function is $R$-invariant. Since this construction is made for a general subgroup $R$ we can also use it later to reduce to the case $G \cong \Z^d$ when the graph is locally finite.

\begin{mydef}\label{factorsDef}
Let $(b,c)$ be a graph on $X$ with a group action $G$ such that $b$ and $c$ are $G$-invariant. Let $R \subseteq G$ be a normal subgroup. Then we define
$$X_R := \{Rx \mid x \in X\}, \quad b_R : X_R \times X_R \to [0, \infty), \quad  c_R : X_R \to \R$$
with
$$b_R(Rx,Ry) := \sum_{r \in R} b(x, ry) \quad \text{ and } \quad c_R(Rx) := c(x).$$
Here $Rx$ denotes the set $\{rx \mid r \in R\}$ for every $x \in X$.
We call $X_R, b_R$ and $c_R$ the space, graph and potential \emph{factored by} $R$ respectively. Moreover, for $Rg \in G / R$ and $Rx \in X_R$ define
$$Rg \: Rx := R(gx) \in X_R.$$
Note that the definition of $b_R$ may result in diagonal terms $b_R(Rx,Rx) > 0$ as defining it this way makes some calculations simpler. However, this does affect the operator $H_{b_R,c_R}$. 
\end{mydef}

\begin{lemma}
The notions in \Cref{factorsDef} are well-defined. Moreover, $b_R$ defines a graph on $X_R$ and the multiplication of an element in $G /R$ with an element in $X_R$ defines a group action of $G /R$ on $X_R$.

\begin{proof}
For $b_R$ and $c_R$ we have to show that the definitions are independent of the representative. To this end, let $x_1, x_2, y_1, y_2 \in X$ such that $x_1 = r_xx_2$ and $y_1 = r_yy_2$ with $r_x, r_y \in R$. Then we have
$$b_R(Rx_1,Ry_1) = \sum_{r \in R} b(x_1, ry_1) = \sum_{r \in R} b(r_x^{-1}r_xx_2, r_x^{-1}rr_yy_2) = \sum_{r' \in R} b(x_2, r'y_2) = b_R(Rx_1,Ry_2),$$
where we used $G$-invariance of $b$ in the first step and the substitution $r' = r_x^{-1}rr_y$ in the second step. Moreover, for $c_R$ we get
$$c_R(Rx_1) = c(x_1) = c(r_xx_2) = c_R(Rx_2),$$
where we used $G$-invariance of $c$. To show that $b_R$ is finite and defines a graph on $X_R$ we have to show that the sum converges and that $b_R$ is locally summable which we will do with one calculation. Indeed, let $x \in X$, then
$$\deg_{b_R}(Rx) = \sum_{Ry \in X_R} b_R(Rx,Ry) = \sum_{Ry \in X_R} \sum_{r \in R} b(x,ry) = \sum_{y' \in X} b(x,y') = \deg_b(x),$$
which is finite as $b$ is a graph. We conclude that $b_R$ is a well-defined graph and $c_R$ is a well-defined potential. It remains to show that the multiplication operation from the group is a group action. Indeed, let $g_1, g_2 \in G$ with $g_1 = r_gg_2$ and $x_1,x_2 \in X$ with $x_1 = r_xx_2$. Then we have 
$$Rg_1 \: Rx_1 = Rg_1x_1 = R r_gg_2r_xx_2 = Rr_g \underbrace{g_2 r_xg_2^{-1}}_{\in R} g_2 x_2 = Rg_2x_2$$
showing the well-definedness of the group multiplication. Finally we see
$$Re \: Rx = Rx$$
and
$$(Rg Rh) \: Rx  = Rgh \: Rx = Rghx = Rg \: Rhx = Rg \: (Rh \: Rx)$$
for all $g,h \in G$ and $x \in X$, which shows that the multiplication defined this way is in fact a group action.
\end{proof}
\end{lemma}

\begin{lemma}\label{factorsLem}
On the spaces $C(X)$ and $C(X_R)$, consider the topology of pointwise convergence. Recall that the stabiliser of an element $x \in X$ with respect to the group action is the set $\mathrm{Stab}(x) = \{ g \in G \mid gx = x\}$. The factored space $X_R$, graph $b_R$ and potential $c_R$ satisfy the following properties.
\begin{enumerate}[label=(\alph*)]
    \item The graph $b_R$ and the potential $c_R$ are $G / R$-invariant.
    \item If the group action of $G$ on $X$ is cocompact with a finite set $V \subseteq X$, then the group action of $G / R$ on $X_R$ is cocompact with the finite set $V_R = \{Rv \mid v\in V\}$
    \item If $b$ is connected, then $b_R$ is also connected.
    \item If $b$ is locally finite, then $b_R$ is also locally finite.
    \item If $R$ includes all stabilisers of elements $x \in X$, then the action of $G / R$ on $X_R$ is free.
    \item For a potential $c' := c - \lambda$ we have $c'_R = c_R - \lambda$ for any $\lambda \in \R$.
    \item Let $C(X)^R$ denote the set of $R$-invariant functions on $X$. Then the map 
    $$\Phi : C(X)^R \to C(X_R), \quad \Phi(f)(Rx) = f(x)$$ 
    is a well defined homeomorphism and restricts to a homeomorphism from $C(X)^R~\cap~\mathrm{Dom}(H_{b,c})$ to $\mathrm{Dom}(H_{b_R,c_R})$. Morever, $H_{b,c}$ maps elements from $C(X)^R$ into $C(X)^R$ and the following diagram commutes:
    \vspace{10pt}
         \begin{center}
             \begin{tikzpicture}
              \matrix (m)
                [
                  matrix of math nodes,
                  row sep    = 3em,
                  column sep = 4em
                ]
                {
                  C(X)^R \cap \mathrm{Dom}(H_{b,c})            
                  & C(X_R) \cap \mathrm{Dom}(H_{b_R,c_R}) \\
                  C(X)^R & C(X_R) \\
                };
              \path
                (m-1-1) edge [->] node [left] {$H_{b,c}$} (m-2-1)
                (m-1-1.east |- m-1-2)
                  edge [->] node [above] {$\Phi$} node [below] {$\cong$} (m-1-2)
                (m-1-2) edge [->] node [right] {$H_{b_R,c_R}$} (m-2-2)
                (m-2-1.east) edge [->] node [above] {$\Phi$} node [below] {$\cong$} (m-2-2);
            \end{tikzpicture}
          \end{center}
    \item If every positive harmonic function with respect to $H_{b,c}$ is $R$-invariant, then the map $\Phi$ from (g) restricts to a homeomorphism
    $$\Phi|_{\mathcal{H}^+_{b,c}} : \mathcal{H}^+_{b,c} \to \mathcal{H}^+_{b_R, c_R}.$$
    \end{enumerate}

    \begin{proof}
        \:
        \begin{enumerate}[label=(\alph*)]
        \item Let $g, \in G$ and $x,y \in X$. Then
        \begin{align*}
            b_R(Rgx, Rgy) &= \sum_{r \in R} b(rgx, rgy) = \sum_{r\in R} b(g^{-1}rgx, g^{-1}rgy) \\
            &= \sum_{\underset{r\in R}{r' = g^{-1}rg}} b(r'x, r'y) = \sum_{r \in R} b(rx,ry) = b_R(Rx,Ry),
        \end{align*}
        as $b$ is $G$-invariant and the conjugation $r \mapsto g^{-1}rg$ is a bijection on $R$. Moreover, for $c_R$ we have
        $$c_R(Rgx) = c(gx) = c(x) = c_R(Rx)$$
        as $c$ is $G$-invariant.
        \item Assume that $GV = X$, then we have
        \begin{align*} &(G/ R) V_R = \{Rg Rv \mid Rg \in G / R, v \in V\} \\
        & \: \: = \{ Rgv \mid g \in G, v \in V\} 
        = \{Rx \mid x \in GV\} = X_R,\end{align*}
        showing cocompactness of the group action.
        \item Assume that $b$ is connected and let $x,y \in X$. There exists a path $x= z_1 \sim \dots \sim z_n = y$ connecting $x$ and $y$ in $X$ via $b$. Observe that whenever $b(z_i,z_{i+1}) > 0$, then 
        $$b_R(Rz_i, Rz_{i+1}) = \sum_{r \in R} b(z_i, rz_{i+1}) > 0$$
        as $e \in R$. Hence the path $Rx = Rz_1 \sim \dots \sim Rz_n = Ry$ connects $Rx$ and $Ry$ in $X_R$ via $b_R$, which shows the desired connectedness.
        
        \item Assume that $b$ is locally finite. Let $x \in X$. Then the set of neighbours of $Rx$ with respect to $b_R$ is given by
        $$\{ Ry \in X_R \mid b_R(Rx, Ry) > 0 \} = \{Ry \in X_R \mid b(x,ry) > 0 \text{ for some }r \in R\},$$
        which must be finite, as there are only finitely many elements $ry$ with $r \in R, y \in X$ such that $b(x,ry)>0$ due to local finiteness of $b$. This shows that $b_R$ is locally finite.
        \item Assume that $R$ includes all stabilisers. Let $Rg \in G / R$ and $Rx \in X_R$ with $RgRx = Rgx = Rx$. It follows that there exists $r \in R$ such that $rg \in \mathrm{Stab}(x) \subseteq R$, and hence $Rg = Rrg = Re$ follows.
        \item For $\lambda \in \R$ and $x \in X$ we have $c'_R(Rx) = c'(x) = c(x) - \lambda = c_R(Rx) - \lambda$.
        
        \item It is clear that $\Phi$ is well-defined as the functions in its domain are $R$-invariant. Moreover, it is easy to see that for $f' \in C(X_R)$ putting $\Phi^{-1}(f')(x) = f'(Rx)$ defines an inverse of $\Phi$. Continuity of $\Phi$ and its inverse also follows immediately from the definitions.
        Furthermore, we see that for every $f \in C(X)^R$ we have
        \begin{align*}\sum_{Ry \in X_R} b_R(Rx,Ry) |\Phi(f)(Ry)| = \sum_{Ry \in X_R} \sum_{y' \in Ry}b(x,y')|f(y')| 
        = \sum_{y \in X} b(x,y)|f(y)| ,\end{align*}
        so $f \in C(X)^R \cap \mathrm{Dom}(H_{b,c})$ is equivalent to $\Phi(f) \in \mathrm{Dom}(H_{b_R,c_R})$. Now let $f \in C(X)^R$ and $r \in R$. Then by \Cref{HGinv} we know that $H_{b,c}$ is $G$-invariant, hence
        $$T_rH_{b,c}f = H_{b,c}T_rf = H_{b,c}f,$$
        so $H_{b,c}$ maps $C(X)^R$ to itself. Finally, for $f \in C(X)^R$ and $x \in X$, we have
        \begin{align*}\Phi^{-1}(H_{b_R,c_R}\Phi(f))(x) &= \sum_{Ry \in X_R} b_R(Rx,Ry)(\Phi(f)(Rx) - \Phi(f)(Ry)) + c(Rx)\Phi(f)(Rx) \\
        &= \sum_{Ry \in X_R} \sum_{y' \in Ry} b(x,y')(f(x) - f(y')) + c(x)f(x) \\
        &= \sum_{y \in X} b(x,y)(f(x) - f(y)) + c(x)f(x) \\
        &= H_{b,c}f(x),
        \end{align*}
        showing that the diagram commutes.
        \item Assume that $\mathcal{H}^+_{b,c} \subseteq C(X)^R$. Then we may restrict $\Phi$ to $\mathcal{H}_{b,c}^+$. We observe that $\Phi$ and $\Phi^{-1}$ map positive functions to positive functions. Moreover, the commutativity of the diagram in (f) yields that
        $$H_{b,c}\Phi(f) = \Phi(H_{b,c}f) = \Phi(0) = 0$$
        for any harmonic $f \in C(X)^R$, hence $\Phi|_{\mathcal{H}^+_{b,c}}$ maps into $\mathcal{H}_{b_R,c_R}$. Similarly, one can see that $\Phi^{-1}$ maps harmonic functions into harmonic functions, so the restriction is also surjective, which shows that it is a homeomorphism together with (g).
        \end{enumerate}
    \end{proof}
\end{lemma}

This construction combined with \Cref{thm1} now allows us to study graphs with abelian group actions instead of general nilpotent group actions. This is achieved by factoring the group, space and graph by the commutator subgroup. This way, we retain all the information about the positive harmonic functions when the operator is invariant under the group shift. 

\section{Structural results for locally finite graphs}
In this section, we restrict ourselves to locally finite graphs. In this case we can characterise the extreme points of $\mathcal{K}$ as exactly the multiplicative elements. Moreover, we will study the structure of the set of positive multiplicative generalised eigenfunctions of the Schrödinger operator, which gives us some insight into the set of all positive generalised eigenfunctions. These results can be seen as a discrete analogue to the results in \cite[sections 4-6]{linPinchover} on differential operators on manifolds.

\subsection{Characterizing the extreme points of $\mathcal{K}$}
We have seen in the previous section that every harmonic extreme point of $\mathcal{K}$ is multiplicative. Our aim here is to show that the converse holds as well when the graph is locally finite. Denote the set of positive, normalised multiplicative and harmonic functions on a locally finite graph $(b,c)$ by

$$\mathcal{M} = \{f \in \mathcal{K} \mid f \text{ is multiplicative}\}.$$

Note that this set only contains harmonic functions, since $\mathcal{K}$ only contains harmonic functions by local finiteness of the graph. Indeed, consider a convergent sequence $(f_n)_{n\in \N}~\in~\mathcal{H}^+$ with $f_n(x_0) =1$. Then taking the limit outside of the finite sum yields

\begin{align*}
    H_{b,c}(\lim_{n \to \infty}f_n)(x) &= \sum_{y \sim x} b(x,y) \lim_{n \to \infty}(f_n(x) - f_n(y)) + c(x)\lim_{n \to \infty}f_n(x) \\
    &= \lim_{n \to \infty}\left( \sum_{y \sim x} b(x,y)(f_n(x) - (f_n(y)) + c(x) f_n(x)\right) = \lim_{n \to \infty}H_{b,c}f_n(x) = 0,
\end{align*}
hence $\{f \in \mathcal{H}^+ \mid f(x_0) = 1 \}$  is closed.
Moreover, for a sequence $f_n \in \mathcal{M}$ converging pointwise to $f \in \mathcal{K}$ we have

$$T_gf = \lim_{n \to \infty} T_gf_n = \lim_{n \to \infty} f_n(gx_0)f_n = f(gx_0)f,$$
so $f$ is also multiplicative and therefore an element of $\mathcal{M}$. It follows that $\mathcal{M}$ is compact as a closed subset of the compact $\mathcal{K}$. We may now formulate the second main theorem of the paper. As a corollary we also obtain a Liouville property if the potential vanishes. 

\begin{theorem}\label{MKisexK}
Let $(b,c)$ be a locally finite, connected graph. Let $G$ be a nilpotent group acting cocompactly on $X$ such that $H_{b,c}$ is $G$-invariant. Then
$$\mathcal{M} = \mathrm{ex} \: \mathcal{K}.$$
\end{theorem}

Before proving this theorem, we observe that it yields the following Liouville property as a corollary.

\begin{corollary}[Liouville property]
Let $b$ be a weight function such that $(b,0)$ is a locally finite, connected graph. Let G be a nilpotent group
acting cocompactly on X such that $H_{b,0}$ is G-invariant. Then every bounded harmonic function is constant.
\begin{proof}
We first observe that, for a vanishing potential, the constant function $1$ is an element of $\mathcal{M}$ as it is positive, harmonic and $G$-invariant. Let $f \in C(X)$ be bounded and harmonic. Let $C > 0$ such that $-C < f < C$. Then $h := \frac{1}{2C}(C + f)$ is harmonic and satisfies $0 < h < 1$. It follows that the functions $h/h(x_0)$ and $(1-h)/(1-h(x_0))$ are elements of $\mathcal{K}$. Since $1 \in \mathcal{M}$ it follows from \Cref{MKisexK} that it is an extreme point of $\mathcal{K}$ and hence, the equation
$$ 1 = h(x_0) \cdot h/h(x_0) + (1 - h(x_0)) \cdot (1-h)/(1-h(x_0))$$
shows that $h/h(x_0) = 1 $, so $h$ and therefore $f$ is constant.
\end{proof}
\end{corollary}

We will prove the \Cref{MKisexK} in two steps. The first inclusion follows immediately from \Cref{thm1} of the previous section, and the second one will be proven later. First, we observe the following lemma, which holds even for graphs which are not locally finite. It is analogous to \cite[Lemma 5.2 (i)]{linPinchover}, and the proof is similar.

\begin{lemma}\label{supersubharm}
Let $(b,c)$ be a graph on $X$ with a cocompact group action $G$, let $f_1, f_2 \in \mathrm{Dom}(H_{b,c})$ be strictly positive, normalised and multiplicative with the same associated homomorphism $\gamma_{f_1} = \gamma_{f_2}$. If $f_1$ is superharmonic and $f_2$ is subharmonic, then $f_1 = f_2$ and both are harmonic.
\begin{proof}
Let $V$ be a fundamental domain of the group action. Put $\gamma := \gamma_{f_1} = \gamma_{f_2}$. As $f_1$ and $f_2$ are both positive, we have
$$C := \min_{v \in V,} \frac{f_1(v)}{f_2(v)} > 0$$
and for $x \in X$ we can write $x = gv$ for $g \in G$ and $v \in V$ and obtain
$$f_1(x) - Cf_2(x) = \gamma(g)(f_1(v) - Cf_2(v)) \geq \gamma(g)\left((f_1(v) - \frac{f_1(v)}{f_2(v)} f_2(v)\right) = 0,$$
so the function $h := f_1 - C f_2$ is nonnegative. Moreover, $h$ is superharmonic as $f_1$ is superharmonic and $f_2$ is subharmonic, so $h \in \{ 0 \} \cup \mathcal{S}^+$. Furthermore, $h$ is not strictly positive since $h(v) = 0$ when $C = f_1(v) /f_2(v)$, so $h = 0$ follows from \Cref{harCor1} of the Harnack inequality. We conclude that $f_1 = Cf_2$. Since both functions are normalised at $x_0$ we have $C = 1$ and the statement follows.
\end{proof}
\end{lemma}

In particular, whenever $f_1, f_2 \in \mathcal{M}$ and $\gamma_{f_1} = \gamma_{f_2}$, this lemma shows that $f_1 = f_2$. With this at hand, we can prove the following lemma and then the proof of \Cref{MKisexK} follows readily. Like the previous one, \Cref{MKlem} is analogous to \cite[Lemma 5.4]{linPinchover}, and the proof is similar.

\begin{lemma}\label{MKlem}
Let $(b,c)$ be a locally finite and connected graph on $X$ with a cocompact group action $G$ and a fixed $x_0 \in X$. Then
$$\mathcal{M} = \mathrm{ex} \: \overline{\mathrm{conv}} \: \mathcal{M}.$$

\begin{proof}
First, since $\mathcal{M} \subseteq \mathcal{K}$, it follows that $\overline{\mathrm{conv}} \mathcal{M} \subseteq \mathcal{K}$ is compact. Then, by the Krein-Milman Theorem, we have
$$\mathrm{ex} \: \overline{\mathrm{conv}} \mathcal{M} \subseteq \overline{\mathcal{M}} = \mathcal{M}.$$
Assume for the sake of contradiction that equality does not hold, i.e. there exists $h \in \mathcal{M} \setminus \mathrm{ex} \: \overline{\mathrm{conv}} \mathcal{M}$. Then, when applying Choquet's theorem to $h$, we may extend the obtained measure $\mu$ to $\mathcal{M}$ such that
$$\int\displaylimits_\mathcal{M}\varphi(k) \d \mu(k)=\varphi(h) = \int\displaylimits_\mathcal{M}\varphi(k) \d \delta_h(k) $$
holds for all continuous linear functionals $\varphi$ on $\mathcal{M}$, where $\delta_h$ is the Dirac measure at $h$. We note that $\mu$ and $\delta_0$ are different, since they have disjoint support, but their integrals of continuous linear functionals coincide by the equation above. Let
$$A \defeq \mathrm{lin}\{\varphi_g : \mathcal{M} \to \R, \: f \mapsto \gamma_f(g) = f(gx_0) \mid g \in G\}.$$
Then $A$ is a subspace of $C(\mathcal{M})$ by construction, but it is even a subalgebra, as 
$$\varphi_g(f)\varphi_{l}(f) = f(gx_0)f(lx_0) = f(glx_0) = \varphi_{gl}(f)$$
holds for any $f \in \mathcal{M}$ and $g,l \in G$. Moreover, $\varphi_e \in A$ is constant, as $\varphi_e(f) = f(x_0) = 1$ for all $f \in \mathcal{M}$. Finally, for $f_1 \neq f_2 \in \mathcal{M}$, we have $\gamma_{f_1} \neq \gamma_{f_2}$ by \Cref{supersubharm}, and hence there exists $g \in G$ with
$$\varphi_g(f_1) = \gamma_{f_1}(g) \neq \gamma_{f_2}(g) =\varphi_g(f_2).$$
We have shown that $A$ is a subalgebra of $C(\mathcal{M})$ which separates the points and contains the constants, so it is dense in $C(\mathcal{M})$ with respect to the supremum norm by the theorem of Stone-Weierstrass.

It follows that $\delta_0(\varphi) = \mu(\varphi)$ holds for all $\varphi \in C(\mathcal{M})$, since $\delta_h$ and $\mu$ are probability measures and they coincide on $A$. We conclude that $\delta_h = \mu$ by the uniqueness of the Riesz-Markov representation theorem, which is a contradiction. 
\end{proof}
\end{lemma}

\begin{proof}[Proof of \Cref{MKisexK}]
From \Cref{harmEx} in the last chapter, we know that $\mathrm{ex} \: \mathcal{K} \subseteq \mathcal{M}$ and therfore $\mathcal{K} = \overline{\mathrm{conv}} \mathcal{M}$ by the Krein-Milman theorem and the fact that $\mathcal{M}$ is a subset of the convex and compact set $\mathcal{K}$.  Combining this with \Cref{MKlem}, we conclude
\begin{equation*}
    \mathcal{M} = \mathrm{ex} \: \overline{\mathrm{conv}} \mathcal{M} = \mathrm{ex} \: \mathcal{K}. \qedhere
\end{equation*}
\end{proof}

\subsection{Multiplicative generalised eigenfunctions}
So far we have only studied the positive solutions of a single operator $H = H_{b,c}$. In this subsection, we will consider positive generalised eigenfunctions, i.e. solutions of $H - \lambda$ for $\lambda \in \R$, and these solutions will be called $\lambda$-harmonic. Since $b$ is assumed to be locally finite, we can apply \Cref{MKisexK} for each $\lambda \in \R$ to the operator $H - \lambda = H_{b, c-\lambda}$, and in view of this theorem and Choquet's theorem we are particularly interested in studying the multiplicative positive $\lambda$-harmonic functions. 

We denote by $\mathcal{H}^+_\lambda, \mathcal{S}^+_{1, \lambda}, \mathcal{K}_\lambda$ and $ {\mathcal{M}_\lambda}$ the sets previously defined as $\mathcal{H}^+, \mathcal{S}^+_1, \mathcal{K}$ and $\mathcal{M}$ respectively, but for the graph $(b,c-\lambda)$, so for example $\mathcal{H}^+_\lambda$ is the set of positive $\lambda$-harmonic functions. Moreover, for $I \subseteq \R$ we write $\mathcal{H}^+_I, \mathcal{K}_I$ and ${\mathcal{M}_I}$ for the unions over $\lambda \in I$ of $\mathcal{H}_\lambda^+, \mathcal{K}_\lambda$ and ${\mathcal{M}_\lambda}$ respectively. Since these sets only consist of $\lambda$-harmonic functions, the unions over different parameters are always disjoint. Then the third main theorem of the paper is the following, which gives a better understanding of the structure of the family of sets $(\mathcal{M}_\lambda)_{\lambda \in \R}$. 

\begin{theorem}\label{thm3}
 Let $(b,c)$ be a locally finite and connected graph on $X$ with a cocompact and nilpotent group action $G$ such that $H_{b,c}$ is $G$-invariant. Then there exists $\lambda_0 > 0$ and $d \in \N_0$ such that the family of sets $\mathcal{M}_\lambda$ for $\lambda \in \R$ satisfies the following properties. 
 \begin{enumerate}[label=(\alph*)]
     \item The set ${\mathcal{M}_\lambda}$ is empty for $\lambda > \lambda_0$.
     \item The set ${\mathcal{M}_{\lambda_0}}$ consists of exactly one point.
     \item The set $\mathcal{M}_\lambda$ is homeomorphic to a $(d-1)$-sphere (empty if $d =0$) for $\lambda < \lambda_0$.
     \item The set ${\mathcal{M}_{(\lambda,\lambda_0]}}$
     is homeomorphic to an open $d$-ball for any  $\lambda \in [-\infty, \lambda_0)$. 
 \end{enumerate}
\end{theorem}

This theorem is similar to \cite[Theorem 5.15]{linPinchover}, however the result here is a purely topological one. The situation here is also somewhat simplified, since we can always identify the space of group homomorphisms from $G$ to $\R^+$ with $\R^n$. We achieve this by reducing everything to the case of $G = \Z^d$. Indeed, the following lemma allows us to do this.

\begin{lemma}\label{lemZd}
Let $(b,c)$ be a locally finite connected graph on $X$ with a nilpotent cocompact group action $G$ such that $H_{b,c}$ is $G$-invariant. Then there exists a normal subgroup $R \subseteq G$ such that every $f \in \mathcal{H}_{\R}^+$ is $R$-invariant, $G / R \cong \Z^d$ for some $d \in \N_0$ and $R$ includes all stabilisers of elements $x \in X$.
\end{lemma}

\begin{proof}
We define
$$R := \bigcap_{f \in \mathcal{M}_{\R}} \mathrm{ker} \: \gamma_f,$$
which is a normal subgroup as an intersection of kernels of group homomorphisms. We see that every $f \in \mathcal{M}_{\R}$ is $R$-invariant, and by \Cref{CorChoq} it follows that every $f \in \mathcal{H}_{\R}^+$ is $R$-invariant. Moreover, when $g \in \mathrm{Stab}(x)$ for $x \in X$ we have $$\gamma_f(g) = \frac{f(gx)}{f(x)} = \frac{f(x)}{f(x)} = 1$$
for all $f \in \mathcal{M}_{\R}$, so $g \in R$ and it follows that $R$ includes all the stabilisers. 
It is left to show that $G / R \cong \Z^d$. We do this by showing that it is abelian, torsion free and finitely generated.
First, observe $[G,G] \subseteq R$ by \Cref{thm1}, hence $G / R$ is abelian. Moreover, if $Rg \in G / R$ has finite order $n$, then $\gamma_f(g^n) = 1$ for all $f \in \mathcal{M}_{\R}$ which implies $\gamma_f(g) = 1$ , hence $Rg = Re$ and it follows that $G / R$ is torsion-free.

Finally, let $V$ be a fundamental domain of the group action and define
$$S := \{Rs \in G / R \mid b(v,sw) > 0 \text{ for some } v,w \in V\}.$$
The set $S$ is finite, as $V$ is finite, $b$ is locally finite and $R$ contains all stabilisers. Given $g \in G$ it follows from connectedness of the graph together with cocompactness of the $G$-action that there exists a path
$$x_0 = g_1v_1 \sim \dots \sim g_nv_n = gx_0$$
where $g_i \in G$ and $v_i \in V$. By $G$-invariance of $b$ we have $b(v_i, g_i^{-1}g_{i+1}v_{i+1}) > 0$ and hence putting $s_i = g_i^{-1}g_{i+1}$ we obtain
$$Rg = Rs_1 \dots Rs_{n-1}$$
where $Rs_i \in S$, as $R$ also includes the stabiliser of $x_0$. It follows that $G / R$ is finitely generated by $S$.
We conclude that $G / R$ is abelian, torsion free and finitely generated, so $G / R \cong \Z^d$ for some $d \in \N_0$.
\end{proof}

Our first goal in this section is identifying the total set $\mathcal{M}_\R$ of the family of multiplicative normalised eigenfunctions with $\R^d$. Indeed, since a homomorphism from $(\Z^d, +)$ to $(\R, +)$ is uniquely determined by its values on the $d$ generators $e_1, \dots e_d$ of $\Z^d$, we may identify a vector $\alpha \in \R^d$ with the homomorphism 

$$ \alpha : \Z^d \to \R, \quad \alpha(z_1, \dots, z_d) \defeq \langle \alpha, z \rangle = \sum_{k=1}^d \alpha_k z_k,$$ 
where $\alpha_k$ denotes the $k$th component of the vector $\alpha$. Clearly this identification is an isomorphism of topological vector spaces, so in this way it is only a slight abuse of notation to denote both objects by $\alpha$.

This allows us to define
$$\rho : \mathcal{M}_{\R} \to \R^d, \quad f \mapsto \alpha_f \defeq \log \circ \gamma_f,$$
which associates to each multiplicative function its character, composed with the logarithm. The result $\alpha_f$ is a homomorphism from $(\Z^d, +)$ into $(\R,+)$, so it can be viewed as an element of $\R^d$ by the discussion above.
From \Cref{supersubharm} it follows that $\rho$ is actually injective, as $f \mapsto \gamma_f$ is injective. However, we want to show that $\rho$ is also surjective, and that it is in fact a homeomorphism.

In the following, we show surjectivity of $\rho$ by explicitly constructing an inverse function. For $\alpha \in \R^d$, define the matrix $Q_\alpha$ by

$$(Q_\alpha)_{v,w} \defeq \sum_{z \in \Z^d} b(v,zw) \exp(\langle \alpha, z \rangle )$$
for $v, w \in V$ with $ v \neq w$, and
$$(Q_\alpha)_{v,v} =  \sum_{z \in \Z^d} b(v,zv) \exp(\langle \alpha, z \rangle ) + \max_{w \in V} \deg(w) - \deg(v)$$
for all $v \in V$. We see that $Q_\alpha$ has nonnegative entries, and it is irreducible as the graph is connected. Indeed, given $v,w \in V$ we can find a path $v = z_1w_1 \sim z_2w_2 \sim \dots \sim z_nw_n = w$ in $b$ and then we estimate
\begin{align*}(Q_\alpha^{n-1})_{v,w} \geq \prod_{i=1}^{n-1} (Q_\alpha)_{w_i,w_{i+1}} &\geq \prod_{i=1}^{n-1} b(w_i, (z_{i+1} - z_i)w_{i+1})\exp(\langle \alpha, z_{i+1} - z_i \rangle ) \\
&= \exp(\langle \alpha, z_n - z_1 \rangle) \prod_{i=1}^{n-1} b(z_iw_i,z_{i+1}w_{i+1}) > 0. \end{align*}

Therefore, we can apply the Perron-Frobenius theorem to $Q_\alpha$ and we obtain the following connection of the matrices $Q_\alpha$ to $\rho$. This lemma  is somewhat analogous to \cite[Corollary 4.9]{linPinchover}, but the situation here is slightly simplified since we can use the theorem of Perron-Frobenius.

\begin{lemma}\label{rhobij}
    The function $\rho$ is bijective. Its inverse is given by
    $\rho^{-1}(\alpha)(zv) = \exp(\langle \alpha, z \rangle) \varphi_{\alpha}(v)$, where $\varphi_\alpha$ denotes the Perron-Frobenius eigenvector of $Q_\alpha$. Moreover, $\rho^{-1}(\alpha) \in \mathcal{M}_\lambda$ where $\lambda = \max_{v \in V} \deg(v) - \theta(\alpha)$ and $\theta(\alpha)$ denotes the Perron-Frobenius eigenvalue of $Q_\alpha$.
    \begin{proof}
        
Let $\alpha \in \R^d$ and consider the function 
$$h_\alpha : X \to \R, \quad h_\alpha(zv) = \exp(\langle \alpha, z \rangle )$$
for $v \in V$ and $z \in \Z^d$, which is well-defined since the group action is free. Then, the associated additive homomorphism of $h_\alpha$ is $\alpha$, so a function $f \in \mathcal{M}_\R$  satisfies $\rho(f) = h_\alpha$ if and only if  it has the form $f = h_\alpha \varphi$  where $\varphi$ is a $G$-invariant function. This leads us to consider the following operator. 
$$H_{h_\alpha} : C(X) \to C(X), \quad H_{h_\alpha}f \defeq \frac{1}{h_\alpha} H (h_\alpha f).$$
This operator is usually called \emph{ground state transform} of $H$ with respect to $h_\alpha$ in the literature. It is defined in such a way that it commutes with the group shift. Indeed, for $f \in C(X)$ and $g \in G$ we have
\begin{align*} H_{h_\alpha}T_gf &= \frac{1}{h_\alpha}H(h_\alpha T_gf) \\
&= \frac{1}{h_\alpha}H(\exp(-\alpha(z)\rangle) T_g(h_\alpha f)) \\
&= \frac{1}{\exp(\alpha(z)) h_\alpha}T_gH(h_\alpha f) \\
&= T_g H_{h_\alpha},\end{align*}
since $h_\alpha$ is multiplicative with character $\exp \circ \alpha$ and $H$ commutes with the group shift. In particular, $H_{h_\alpha}$ maps $G$-invariant functions to $G$-invariant functions. Moreover, we observe that for a $G$-invariant function $\varphi$, the product $f = h_\alpha \varphi$ is positive and $\lambda$-harmonic if and only if $\varphi$ is positive and
$$(H_{h_\alpha} -\lambda)\varphi = 0.$$
In other words, there is a one-to-one correspondence of elements $f \in \rho^{-1}(\{\alpha\})$ and positive eigenvectors $\varphi$ of the restriction of $H_{h_\alpha}$ to $C(X)^G$, the space of $G$-invariant functions. Now the space $C(X)^G$ is isomorphic to $C(V)$, as a function on $V$ can be extended periodically to $C^G(V)$, and this map is surjective. Under this isomorphism, the matrix entries of the restriction $H_{h_\alpha}|_{C(X)^G}$ are
$$H_{h_{\alpha}}1_{Gw}(v) = \frac{1}{h_{\alpha}(v)}\sum_{g \in G}b(v,gw)(- h_{\alpha}(gw)) = -\sum_{z \in \Z^d} b(v,zw)\exp(\langle \alpha, z \rangle)$$
for the off-diagonal entries corresponding to $v,w \in V$, and 
$$H_{h_\alpha}1_{Gv}(v) = \frac{1}{h_\alpha(v)} \left(\sum_{g \in G} b(v,gw)(h_\alpha(v) - h_\alpha(gv)) + c(v)h_{\alpha}(v)\right) = \deg(v) - \sum_{z \in \Z^d} b(v,zv) \exp(\langle \alpha, z \rangle).$$
for the diagonal entries corresponding to $v \in V$. here $1_{Gv}$ denotes the characteristic function of the orbit of $v$ with respect to the group action. Comparing this to the matrix $Q_\alpha$, we see that $$H_{h_\alpha}|_{C(X)^G} = \max \deg(v)I - Q_\alpha,$$
so the two matrices have the same eigenspaces. Then the elements $f \in \rho^{-1}(\{\alpha\})$ are given by $f = h_\alpha \varphi$ where $\varphi$ is a positive eigenvector of $H_{h_\alpha}|_{C(X)^G}$. However, $Q_\alpha$ has exactly the same eigenvectors as $H_{h_\alpha}|_{C(X))^G}$, and the Perron-Frobenius eigenvector $\varphi_\alpha$ is the unique positive eigenvector of $Q_\alpha$. It follows that the function $\rho$ is indeed bijective and the inverse is given by
$$\rho^{-1}(\alpha)(zv) = h_\alpha(zv)\varphi_\alpha(v) = \exp(\langle \alpha, z \rangle ) \varphi_\alpha(v),$$
where $\varphi_\alpha$ is the Perron-Frobenius eigenvector of $Q_\alpha$. 

For the moreover part, observe that the eigenvalue of $Q_\alpha$ associated to $\varphi_\alpha$ is the Perron-Frobenius eigenvalue $\theta(\alpha)$, so the eigenvalue of $H_{h_\alpha}|_{C(X)^G}$ to the same vector is
$$\lambda = \max_{v \in V} \deg(v) - \theta(\alpha).$$
By construction of $H_{h_\alpha}$, it follows that $f$ is $\lambda$-harmonic for the same $\lambda$.
    \end{proof}
\end{lemma}

We can now identify the total space $\mathcal{M}_\R$ with $\R^d$, and we will see that $\rho$ is actually a homeomorphism in \Cref{hhomeo} below. This identification allows us to partition the space $\R^d$ into the following family of sets. For $\lambda \in \R$ and $I \subseteq \R$, define
$$\mathcal{A}_\lambda := \rho(\mathcal{M}_\lambda), \quad \mathcal{A}_I := \bigcup_{\lambda' \in I} \mathcal{A}_{\lambda'} = \rho(\mathcal{M}_I).$$
In the following, we will study the family $(\mathcal{A}_\lambda)_{\lambda \in \R}$, and then we will be able to transfer all of the topological properties of this family back to $(\mathcal{M}_\lambda)_{\lambda \in \R}$ for the proof of \Cref{thm3}. The sets $\mathcal{A}_\lambda$ are much easier to work with, since there is a vector space structure on $\R^d$. Indeed, many useful properties will follow from the concavity of the following function.
$$\Lambda : \mathcal{A}_\R \to \R, \quad \alpha \in \mathcal{A}_\lambda \mapsto \lambda.$$
This is simply the base projection map for the family $(\mathcal{A}_\lambda)_{\lambda \in \R}$. However, by \Cref{rhobij} it follows, that we have 
$$\Lambda(\alpha) = \max_{v \in V} \deg(v) - \theta(\alpha),$$ where $\theta(\alpha)$ is the Perron-Frobenius eigenvalue of $Q_\alpha$. From this identity we immediately see that $\Lambda \leq \max_{v \in V} \deg(v)$, as Perron-Frobenius eigenvalues are always positive. This means that there are no multiplicative $\lambda$-harmonic functions for $\lambda$ above a certain threshold. Since we want a precise threshold, we define
$$ \lambda_0 \defeq \sup_{\alpha \in \R^d}  \Lambda(\alpha) = \sup \{\lambda \in \R \mid \mathcal{A}_\lambda \neq \emptyset \} = \sup \{\lambda \in \R \mid \mathcal{M}_\lambda \neq \emptyset \} $$
Indeed this supremum is a maximum, which we will see in \Cref{Aprop} below. Since $\lambda_0$ is finite and $\mathcal{A}_\lambda$ is empty for $\lambda > \lambda_0$, we will often write $\mathcal{A}_{(\lambda, \lambda_0]}$ and $\mathcal{M}_{(\lambda, \lambda_0]}$ instead of $\mathcal{A}_{(\lambda, \infty)}$ and $\mathcal{M}_{(\lambda, \infty)}$ for $\lambda < \lambda_0$. Of course, $\lambda_0$ can never be $-\infty$, as $\mathcal{A}_{(-\infty, \lambda_0]} = \R^d \neq \emptyset$.

 \begin{remark}
 By definition, $\lambda_0$ is the supremum of all $\lambda \in \R$ for which there exist \emph{any} $\lambda$-harmonic functions, as ${\mathcal{M}_\lambda}$ is empty if and only if $\mathcal{K}_\lambda$ is empty. Then, if the operator $H = H_{b,c}$ is self-adjoint with respect to a measure $m$, the theorem of Agmon-Allegretto-Piepenbrink, see \cite{keller2019criticality}, states that $\lambda_0$ is exactly the bottom of the spectrum of $H$ on  $\ell^2(X,m)$.
 \end{remark}

\begin{remark}
    In order to compute or estimate $\lambda_0$, one can use the representation $\Lambda(\alpha) = \max_{v \in V} \deg(v) - \theta(\alpha)$. Indeed, finding the minimum $\theta_{\mathrm{min}}$ of all Perron-Frobenius eigenvalues of the matrices $Q_\alpha$, one obtains
    $$\lambda_0 = \max_{v \in V} \deg(v) - \theta_{\mathrm{min}}.$$
    Moreover, finding any one Perron-Frobenius eigenvalue $\theta(\alpha)$ of a matrix $Q_\alpha$ already gives a lower bound
    $$ \lambda_0 \geq \max_{v \in V} \deg(v) - \theta(\alpha),$$
    which yields a lower bound on the spectral gap for a self-adjoint operator $H_{b,c}$.
\end{remark}
    
In \Cref{Aprop} we will prove several properties of the sets $\mathcal{A}_\lambda$. The proof, is similar to its analogue \cite[Corollary 5.14]{linPinchover}, but will make use of both representations of $\mathcal{A}_\lambda$ as preimages under $\Lambda$ of $\lambda$ and as images under $\rho$ of $\mathcal{M}_\lambda$. To use these two respectively, we have the following two lemmas. \Cref{lambdaProp} is analogous to \cite[Lemma 5.12]{linPinchover}, where the direct computation which is omitted there is just an application of Hölder's inequality here. \Cref{hhomeo}, on the other hand, is a simple continuity result, which has no clear analogue in \cite{linPinchover}.

\begin{lemma}\label{lambdaProp}
The map $\Lambda : \R^d \to (-\infty, \lambda_0]$ is strictly concave and continuous.

\begin{proof}
Let $\alpha_1, \alpha_2 \in \mathrm{Hom}(\Z^d,\R)$ be different and $t \in (0,1)$. Define $\lambda_i = \Lambda(\alpha_i)$. We know that the functions $f_{\alpha_i} \defeq \rho^{-1}(\alpha_i)$ are normalised, $\lambda_i$-harmonic, positive and multiplicative with the characters $\exp \circ \alpha_i$ for $i \in \{ 1,2\}$. In particular,
$$\sum_{y \in X}b(x,y)f_{\alpha_i}(y) = (\deg(x) - \lambda_i)f_{\alpha_i}(x)$$
holds for all $x \in X$.
Moreover, define $\lambda_t \defeq t\lambda_1 + (1-t)\lambda_2$ and $f_t \in C(X)$ by
$$f_t(x) \defeq (f_{\alpha_1}(x))^t (f_{\alpha_2}(x))^{(1-t)}.$$
Then we see that $f_t$ is normalised and multiplicative with the character $\exp \circ \alpha_t$, where $\alpha_t := t\alpha_1 + (1-t)\alpha_2$ and we also have $f_1 = f_{\alpha_1}$ and $f_2 = f_{\alpha_2}$. Now for a fixed $x \in X$ we may view $X$ as a measure space with the discrete measure $b(x, \cdot)$ and then Hölder's inequality for $p = \frac{1}{t}, q = \frac{1}{1-t}$ yields

\begin{align*}
\sum_{y \in X} b(x,y)f_t(y) &\leq \left(\sum_{y \in X}b(x,y)f_1(y)\right)^t\left(\sum_{y \in X}b(x,y)f_2(y)\right)^{(1-t)} \\
&= (\deg(x) - \lambda_1)^tf_{\alpha_1}(x)^t(\deg(x) - \lambda_2)^{(1-t)} f_{\alpha_2}(x)^{(1-t)} \\
&\leq (\deg(x) - \lambda_t)f_t(x) 
\end{align*}
where the inequality is strict for some $x \in X$. 

Indeed, assume that equality holds for all $x \in X$. Then we can see from the second inequality that $\lambda_1 = \lambda_2 =: \lambda$. From the first inequality, we see that $f_1$ and $f_2$ are linearly dependent when restricted to the support of $b(x,\cdot)$ for all $x \in X$. Let $x \in X$ and $\mu_x > 0$ such that $f_1 = \mu_x f_2$ on the support of $b(x, \cdot)$. It follows from $\lambda$-harmonicity of $f_1$ and $f_2$ that

\begin{align*}
    (\deg(x) - \lambda) f_1(x) &= \sum_{y \in X} b(x,y) f_1(y) \\ &= \sum_{y \in X} b(x,y) \mu_x f_2(y) \\
    &= (\deg(x) - \lambda) \mu_xf_2(x) \end{align*}
and hence $f_1(x)/f_2(x) = \mu_x = f_1(y)/f_2(y)$ for all $x \sim y$. Using the fact that $f_1(x_0)/f_2(x_0) = 1$ and conectedness of the graph, it then follows that $f_1 = f_2$, hence $\alpha_1 = \alpha_2$, which is a contradiction.

We conclude that $f_t$ is $\lambda_t$-superharmonic but not $\lambda_t$-harmonic. Assume that $\Lambda(\alpha_t) \leq \lambda_t$. Then 
$$(H - \Lambda(\alpha_t))f_t = (H - \lambda_t)f_t + (\lambda_t - \Lambda(\alpha_t))f_t \geq 0,$$
so the functions $f_t$ and $\rho^{-1}(\alpha_t)$ are both normalised and multiplicative with the same character $\exp \circ \alpha_t$, but $f_t$ is $\Lambda(\alpha_t)$-superhamonic and not harmonic and $\rho^{-1}(\alpha_t)$ is $\Lambda(\alpha_t)$-harmonic, so $f_t \neq \rho^{-1}(\alpha_t)$. This is a contradiction to \Cref{supersubharm}. It follows that $\Lambda(\alpha_t) > \lambda_t$, showing strict concavity of $\Lambda$. Since $\R^d$ is open it follows from concavity that $\Lambda$ is also continuous.
\end{proof}
\end{lemma}

\begin{lemma}\label{hhomeo}
The map $\rho : {\mathcal{M}_{(-\infty, \lambda_0]}} \to \R^d$ is a homeomorphism. Moreover, for any $\lambda \in \R$ the sets $\mathcal{M}_{[\lambda, \lambda_0]}$ and $\mathcal{A}_{[\lambda, \lambda_0]}$ are compact.

\begin{proof}
Recall that $\rho$ is bijective by \Cref{rhobij}. Moreover, for a sequence $f_n$ converging pointwise to $f$ in $\mathcal{M}_{(-\infty, \lambda_0]}$, we have
 $$\lim_{n \to \infty} \rho(f_n)(g) = \lim_{n \to \infty} \log(\gamma_{f_n}(g)) = \lim_{n \to \infty} \log(f_n(gx_0))  = \log(f(gx_0)) = \rho(f)(g),$$
 so $\rho$ is continuous.
 
 To show continuity of the inverse, we first note that for any $\lambda \in \R$ we have ${\mathcal{M}_{[\lambda, \lambda_0]}} \subseteq \mathcal{S}^+_{1,\lambda}$, as for $\lambda' \geq \lambda$ and $f \in {\mathcal{M}_{\lambda'}} \subseteq \mathcal{H}_{\lambda'}^+$ we have 
 $$(H-\lambda)f = (H - \lambda')f + (\lambda' - \lambda)f  \geq 0,$$ so $f \in \mathcal{S}^+_{1,\lambda}$. Then, since $\mathcal{S}^+_{1,\lambda}$ is compact by \Cref{Scomp}, it follows that ${\mathcal{M}_{[\lambda, \lambda_0]}}$ is compact as it is closed. Indeed, for a sequence, $(f_n)_{n \in \N} \in \mathcal{M}_{[\lambda, \lambda_0]}$ converging pointwise to $f \in C(X)$ we see that $f$ is non-negative, normalised at $x_0$ and multiplicative in exactly the same way that we showed closedness of $\mathcal{M}$ at the beginning of the chapter. Moreover, since $f_n$ is $\lambda_n$-harmonic with $\lambda_n \in [\lambda, \lambda_0]$ for each $n \in \N$, we may choose a subsequence $n_k$ such that $\lambda_{n_k} \to \mu \in [\lambda, \lambda_0].$ Then, by local finiteness of the graph, we see that $f$ is $\mu$-harmonic, and hence $f \in \mathcal{M}_\mu \subseteq \mathcal{M}_{[\lambda, \lambda_0]}$.  It follows that $\mathcal{M}_{[\lambda, \lambda_0]}$ is compact and hence also $\mathcal{A}_{[\lambda, \lambda_0]}$ is compact by continuity of $\rho$. From this we conclude that 
 $$\rho^{-1}|_{\mathcal{A}_{[\lambda, \lambda_0]}} : \mathcal{A}_{[\lambda, \lambda_0]} \to {\mathcal{M}_{[\lambda, \lambda_0]}}$$
 is automatically continuous, in particular the restriction of $\rho^{-1}$ to $\mathcal{A}_{(\lambda, \lambda_0]}$ is continuous. By \Cref{lambdaProp}, $\Lambda$ is continuous, therefore $\mathcal{A}_{(\lambda, \lambda_0]} = \Lambda^{-1}((\lambda, \infty))$ is open, and hence $\rho^{-1}$ itself is continuous at all points in $\mathcal{A}_{(\lambda, \lambda_0]}$. Since $\lambda \in \R$ was arbitrary and $\cup_{\lambda \in \R} \mathcal{A}_{(\lambda, \lambda_0]} = \R^d$, it follows that $\rho^{-1}$ is continuous on $\R^d$, which shows that $\rho$ is a homeomorphism.
\end{proof}
\end{lemma}

\begin{theorem}\label{Aprop}
Denote by $\partial B$ the topological boundary of a set $B \subset \R^d$. If $d \geq 1$ the following statements hold for any $\lambda \in \R$.
\begin{enumerate}[label=(\alph*)]
    \item The sets $\mathcal{A}_{[\lambda, \lambda_0]}$ and $\mathcal{A}_{(\lambda, \lambda_0]}$ are stricly convex.
    \item The set $\mathcal{A}_{(\lambda, \lambda_0]}$ is open and $\mathcal{A}_{[\lambda, \lambda_0]}$ is compact.
    \item $\mathcal{A}_\lambda = \mathrm{ex} \: \mathcal{A}_{[\lambda, \lambda_0]} = \partial \mathcal{A}_{[\lambda, \lambda_0]}$
    \item The set $\mathcal{A}_{\lambda_0}$ consists of exactly one point.
\end{enumerate}

\begin{proof}
We first recall that by definition of $\lambda_0$, we have $\mathcal{A}_{(\lambda,\lambda_0]} = \mathcal{A}_{(\lambda,\infty)}$ and $\mathcal{A}_{[\lambda, \lambda_0]} =\mathcal{A}_{[\lambda, \infty)}$ for all $\lambda \in  [-\infty, \lambda_0)$. We also recall that $\Lambda$ is continuous and stricly concave by \Cref{lambdaProp}.
\begin{enumerate}[label=(\alph*)]
    \item This follows from strict concavity of $\Lambda$, as a convex combination of elements in $\mathcal{A}_{\lambda_1}$ and $\mathcal{A}_{\lambda_2}$ with $\lambda_1, \lambda_2 \geq \lambda$ is mapped to some $\lambda_3$ which is strictly greater than the convex combination of $\lambda_1$ and $\lambda_2$, in particular $\lambda_3 > \lambda$.
    
    \item Since $\Lambda$ is continuous, it follows that $\mathcal{A}_{(\lambda, \lambda_0]} = \Lambda^{-1}((\lambda, \infty))$ is open. Compactness of  $\mathcal{A}_{[\lambda, \lambda_0]}$ was shown in \Cref{hhomeo}.
    
    \item Let $f\in \mathcal{A}_\lambda$ and $f = tf_1 + (1-t)f_2$ with $\Lambda(f_1) \geq \lambda$ and $\Lambda(f_2) \geq \lambda$. Then, by strict concavity of $\Lambda$, we have $f_1 = f_2 = f$, hence $f \in \mathrm{ex}\: \mathcal{A}_{[\lambda, \lambda_0]}$. This shows 
    $$\mathcal{A}_\lambda \subseteq \mathrm{ex} \mathcal{A}_{[\lambda, \lambda_0]} \subseteq \partial \mathcal{A}_{[\lambda, \lambda_0]} \subseteq \mathcal{A}_{[\lambda, \lambda_0]} \setminus \mathcal{A}_{(\lambda, \lambda_0]} = \mathcal{A}_\lambda,$$
    since $\mathcal{A}_{[\lambda, \lambda_0]}$ is closed and $\mathcal{A}_{(\lambda, \lambda_0]}$ is open by (b).
    
    \item We have
    $$\mathcal{A}_{\lambda_0} = \bigcap_{n=1}^\infty \mathcal{A}_{[\lambda_0 - \frac{1}{n}, \lambda_0]},$$
    so using (b), $\mathcal{A}_{\lambda_0}$ is nonempty as a countable intersection of nested nonempty compact sets. Moreover, if $\alpha_1 \neq \alpha_2 \in \mathcal{A}_{\lambda_0}$ then $\Lambda(\frac{1}{2}\alpha_1 + \frac{1}{2}\alpha_2) > \lambda_0$ by strict concavity, which is a contradiction. \qedhere
\end{enumerate}
\end{proof}
\end{theorem}

\begin{proof}[Proof of \Cref{thm3}]
Since $(b,c)$ is locally finite, we may apply \Cref{lemZd} and factor the graph $(b,c)$ by the obtained $R$ as explained in the second chapter. We may apply \emph{all} the parts of \Cref{factorsLem}, as $(b,c)$ is connected, locally finite, the $G$-action is cocompact, $R$ includes all stabilisers, and every $f \in \mathcal{H}_\lambda$ is $R$-invariant. We obtain that the resulting graph $(b_R, c_R)$ on $X_R$ is locally finite and connected, has a cocompact and free group action of $G/R \cong \Z^d$, and factoring $(b, c - \lambda)$ we obtain $(b_R, c_R - \lambda)$. Finally, the homeomorphism
$$\Phi : C(X)^R \to C(X_R)$$
restricts to homeomorphisms identifying the sets $\mathcal{H}_\lambda^+, \mathcal{K}_\lambda$ and $\mathcal{M}_\lambda$ for the graph $(b,c)$ with their respective counterparts for $(b_R,c_R)$.

This means that without loss of generality, we may assume that the group action is a free $\Z^d$-action. Then with the sets $\mathcal{A}_\lambda$ defined as above and $\lambda_0 \defeq \sup \{\lambda \in \R \mid  \mathcal{A}_\lambda \neq \emptyset\}$, it follows from \Cref{hhomeo} that

$$\mathcal{M}_I \cong \mathcal{A}_I$$
for all $I \subseteq \R$. In particular, when $d = 0$ we have $\mathcal{A}_{\R} = \{ 0 \}$, so $\mathcal{M}_{\lambda_0}$ contains a single point, and the other $\mathcal{M}_\lambda$ are empty. For $d \geq 1$ we obtain the following:
\begin{enumerate}[label=(\alph*)]
    \item For $\lambda > \lambda_0$ we have $\mathcal{M}_\lambda \cong \mathcal{A}_\lambda = \emptyset$
    by definition of $\lambda_0$.
    \item By \Cref{Aprop} (d) we know that $\mathcal{A}_{\lambda_0}$
    consists of exactly one point and $\mathcal{M}_{\lambda_0} \cong \mathcal{A}_{\lambda_0}$.
    
    \item For $\lambda < \lambda_0$ we have $\mathcal{M}_\lambda \cong \mathcal{A}_\lambda = \partial \mathcal{A}_{[\lambda, \lambda_0]}$
    by \Cref{Aprop} (c). Moreover, $\mathcal{A}_{[\lambda, \lambda_0]} \subseteq \R^d$ is compact and convex by \Cref{Aprop} (a) and (b), so its boundary is homeomorphic to a $(d-1)$-sphere.
    \item For $\lambda \in (-\infty, \lambda_0)$ we have $\mathcal{M}_{(\lambda, \lambda_0]} \cong \mathcal{A}_{(\lambda, \lambda_0]}$, and $\mathcal{A}_{(\lambda, \lambda_0]}$ is an open and convex subset of $\R^d$ by \Cref{Aprop} (a) and (b), so it is homeomorphic to an open $d$-ball. Also, $\mathcal{M}_{(-\infty, \lambda_0]} \cong \mathcal{A}_{(-\infty, \lambda_0]} = \R^d$.
\end{enumerate}
This finishes the proof.
\end{proof}

\addcontentsline{toc}{section}{Bibliography}
\bibliography{literature.bib}


\begin{thebibliography}{17}
\ifx \bisbn   \undefined \def \bisbn  #1{ISBN #1}\fi
\ifx \binits  \undefined \def \binits#1{#1}\fi
\ifx \bauthor  \undefined \def \bauthor#1{#1}\fi
\ifx \batitle  \undefined \def \batitle#1{#1}\fi
\ifx \bjtitle  \undefined \def \bjtitle#1{#1}\fi
\ifx \bvolume  \undefined \def \bvolume#1{\textbf{#1}}\fi
\ifx \byear  \undefined \def \byear#1{#1}\fi
\ifx \bissue  \undefined \def \bissue#1{#1}\fi
\ifx \bfpage  \undefined \def \bfpage#1{#1}\fi
\ifx \blpage  \undefined \def \blpage #1{#1}\fi
\ifx \burl  \undefined \def \burl#1{\textsf{#1}}\fi
\ifx \doiurl  \undefined \def \doiurl#1{\url{https://doi.org/#1}}\fi
\ifx \betal  \undefined \def \betal{\textit{et al.}}\fi
\ifx \binstitute  \undefined \def \binstitute#1{#1}\fi
\ifx \binstitutionaled  \undefined \def \binstitutionaled#1{#1}\fi
\ifx \bctitle  \undefined \def \bctitle#1{#1}\fi
\ifx \beditor  \undefined \def \beditor#1{#1}\fi
\ifx \bpublisher  \undefined \def \bpublisher#1{#1}\fi
\ifx \bbtitle  \undefined \def \bbtitle#1{#1}\fi
\ifx \bedition  \undefined \def \bedition#1{#1}\fi
\ifx \bseriesno  \undefined \def \bseriesno#1{#1}\fi
\ifx \blocation  \undefined \def \blocation#1{#1}\fi
\ifx \bsertitle  \undefined \def \bsertitle#1{#1}\fi
\ifx \bsnm \undefined \def \bsnm#1{#1}\fi
\ifx \bsuffix \undefined \def \bsuffix#1{#1}\fi
\ifx \bparticle \undefined \def \bparticle#1{#1}\fi
\ifx \barticle \undefined \def \barticle#1{#1}\fi
\bibcommenthead
\ifx \bconfdate \undefined \def \bconfdate #1{#1}\fi
\ifx \botherref \undefined \def \botherref #1{#1}\fi
\ifx \url \undefined \def \url#1{\textsf{#1}}\fi
\ifx \bchapter \undefined \def \bchapter#1{#1}\fi
\ifx \bbook \undefined \def \bbook#1{#1}\fi
\ifx \bcomment \undefined \def \bcomment#1{#1}\fi
\ifx \oauthor \undefined \def \oauthor#1{#1}\fi
\ifx \citeauthoryear \undefined \def \citeauthoryear#1{#1}\fi
\ifx \endbibitem  \undefined \def \endbibitem {}\fi
\ifx \bconflocation  \undefined \def \bconflocation#1{#1}\fi
\ifx \arxivurl  \undefined \def \arxivurl#1{\textsf{#1}}\fi
\csname PreBibitemsHook\endcsname

\bibitem[\protect\citeauthoryear{Choquet and Deny}{1960}]{choquetDeny}
\begin{barticle}
\bauthor{\bsnm{Choquet}, \binits{G.}},
\bauthor{\bsnm{Deny}, \binits{J.}}:
\batitle{Sur l'\'{e}quation de convolution {$\mu =\mu \ast \sigma $}}.
\bjtitle{C. R. Acad. Sci. Paris}
\bvolume{250},
\bfpage{799}--\blpage{801}
(\byear{1960})
\end{barticle}
\endbibitem

\bibitem[\protect\citeauthoryear{Margulis}{1966}]{margulis}
\begin{barticle}
\bauthor{\bsnm{Margulis}, \binits{G.A.}}:
\batitle{Positive harmonic functions on nilpotent groups}.
\bjtitle{Soviet Math. Dokl.}
\bvolume{7},
\bfpage{241}--\blpage{244}
(\byear{1966})
\end{barticle}
\endbibitem

\bibitem[\protect\citeauthoryear{Dynkin and Maljutov}{1961}]{dynkin}
\begin{botherref}
\oauthor{\bsnm{Dynkin}, \binits{E.B.}},
\oauthor{\bsnm{Maljutov}, \binits{M.B.}}:
Random walk on groups with a finite number of generators
\textbf{137},
1042--1045
(1961)
\end{botherref}
\endbibitem

\bibitem[\protect\citeauthoryear{Lin}{1987}]{lin88}
\begin{botherref}
\oauthor{\bsnm{Lin}, \binits{V.Y.}}:
Liouville coverings of complex spaces, and amenable groups
\textbf{132(174)}(2),
202--224
(1987)
\doiurl{10.1070/SM1988v060n01ABEH003163}
\end{botherref}
\endbibitem

\bibitem[\protect\citeauthoryear{Agmon}{1984}]{agmon}
\begin{bchapter}
\bauthor{\bsnm{Agmon}, \binits{S.}}:
\bctitle{On positive solutions of elliptic equations with periodic coefficients
  in {${\bf R}^n$}, spectral results and extensions to elliptic operators on
  {R}iemannian manifolds}.
In: \bbtitle{Differential Equations ({B}irmingham, {A}la., 1983)}.
\bsertitle{North-Holland Math. Stud.},
vol. \bseriesno{92},
pp. \bfpage{7}--\blpage{17}.
\bpublisher{North-Holland, Amsterdam}, \blocation{???}
(\byear{1984}).
\doiurl{10.1016/S0304-0208(08)73672-7}
\end{bchapter}
\endbibitem

\bibitem[\protect\citeauthoryear{Allegretto}{1974}]{allegretto1}
\begin{barticle}
\bauthor{\bsnm{Allegretto}, \binits{W.}}:
\batitle{On the equivalence of two types of oscillation for elliptic
  operators}.
\bjtitle{Pacific J. Math.}
\bvolume{55},
\bfpage{319}--\blpage{328}
(\byear{1974})
\end{barticle}
\endbibitem

\bibitem[\protect\citeauthoryear{Allegretto}{1979}]{allegretto2}
\begin{barticle}
\bauthor{\bsnm{Allegretto}, \binits{W.}}:
\batitle{Spectral estimates and oscillations of singular differential
  operators}.
\bjtitle{Proc. Amer. Math. Soc.}
\bvolume{73}(\bissue{1}),
\bfpage{51}--\blpage{56}
(\byear{1979})
\doiurl{10.2307/2042880}
\end{barticle}
\endbibitem

\bibitem[\protect\citeauthoryear{Allegretto}{1981}]{allegretto3}
\begin{barticle}
\bauthor{\bsnm{Allegretto}, \binits{W.}}:
\batitle{Positive solutions and spectral properties of second order elliptic
  operators}.
\bjtitle{Pacific J. Math.}
\bvolume{92}(\bissue{1}),
\bfpage{15}--\blpage{25}
(\byear{1981})
\end{barticle}
\endbibitem

\bibitem[\protect\citeauthoryear{Piepenbrink}{1974}]{piepenbrink2}
\begin{barticle}
\bauthor{\bsnm{Piepenbrink}, \binits{J.}}:
\batitle{Nonoscillatory elliptic equations}.
\bjtitle{J. Differential Equations}
\bvolume{15},
\bfpage{541}--\blpage{550}
(\byear{1974})
\doiurl{10.1016/0022-0396(74)90072-2}
\end{barticle}
\endbibitem

\bibitem[\protect\citeauthoryear{Piepenbrink}{1977}]{piepenbrink3}
\begin{barticle}
\bauthor{\bsnm{Piepenbrink}, \binits{J.}}:
\batitle{A conjecture of {G}lazman}.
\bjtitle{J. Differential Equations}
\bvolume{24}(\bissue{2}),
\bfpage{173}--\blpage{177}
(\byear{1977})
\doiurl{10.1016/0022-0396(77)90142-5}
\end{barticle}
\endbibitem

\bibitem[\protect\citeauthoryear{Moss and Piepenbrink}{1978}]{piepenbrink1}
\begin{barticle}
\bauthor{\bsnm{Moss}, \binits{W.F.}},
\bauthor{\bsnm{Piepenbrink}, \binits{J.}}:
\batitle{Positive solutions of elliptic equations}.
\bjtitle{Pacific J. Math.}
\bvolume{75}(\bissue{1}),
\bfpage{219}--\blpage{226}
(\byear{1978})
\end{barticle}
\endbibitem

\bibitem[\protect\citeauthoryear{Guivarc'h}{1984}]{guivarch}
\begin{botherref}
\oauthor{\bsnm{Guivarc'h}, \binits{Y.}}:
Sur la repr\'{e}sentation int\'{e}grale des fonctions harmoniques et des
  fonctions propres positives dans un espace riemannien sym\'{e}trique
\textbf{108}(4),
373--392
(1984)
\end{botherref}
\endbibitem

\bibitem[\protect\citeauthoryear{Lin and Pinchover}{1994}]{linPinchover}
\begin{barticle}
\bauthor{\bsnm{Lin}, \binits{V.Y.}},
\bauthor{\bsnm{Pinchover}, \binits{Y.}}:
\batitle{Manifolds with group actions and elliptic operators}.
\bjtitle{Mem. Amer. Math. Soc.}
\bvolume{112}(\bissue{540}),
\bfpage{78}
(\byear{1994})
\doiurl{10.1090/memo/0540}
\end{barticle}
\endbibitem

\bibitem[\protect\citeauthoryear{Keller et~al.}{2021}]{keller}
\begin{bbook}
\bauthor{\bsnm{Keller}, \binits{M.}},
\bauthor{\bsnm{Lenz}, \binits{D.}},
\bauthor{\bsnm{Wojciechowski}, \binits{R.K.}}:
\bbtitle{Graphs and Discrete {D}irichlet Spaces}.
\bsertitle{Grundlehren der mathematischen Wissenschaften}.
\bpublisher{Springer},
\blocation{Berlin}
(\byear{2021}).
\doiurl{10.1007/978-3-030-81459-5}
\end{bbook}
\endbibitem

\bibitem[\protect\citeauthoryear{Phelps}{2001}]{phelps}
\begin{bbook}
\bauthor{\bsnm{Phelps}, \binits{R.R.}}:
\bbtitle{Lectures on {C}hoquet's Theorem}.
\bsertitle{Lecture Notes in Mathematics}.
\bpublisher{Springer},
\blocation{Heidelberg}
(\byear{2001}).
\doiurl{10.1007/b76887}
\end{bbook}
\endbibitem

\bibitem[\protect\citeauthoryear{Tychonoff}{1935}]{schauderTych}
\begin{barticle}
\bauthor{\bsnm{Tychonoff}, \binits{A.}}:
\batitle{Ein {F}ixpunktsatz}.
\bjtitle{Math. Ann.}
\bvolume{111}(\bissue{1}),
\bfpage{767}--\blpage{776}
(\byear{1935})
\doiurl{10.1007/BF01472256}
\end{barticle}
\endbibitem

\bibitem[\protect\citeauthoryear{Keller et~al.}{2020}]{keller2019criticality}
\begin{barticle}
\bauthor{\bsnm{Keller}, \binits{M.}},
\bauthor{\bsnm{Pinchover}, \binits{Y.}},
\bauthor{\bsnm{Pogorzelski}, \binits{F.}}:
\batitle{Criticality theory for {S}chr\"{o}dinger operators on graphs}.
\bjtitle{J. Spectr. Theory}
\bvolume{10}(\bissue{1}),
\bfpage{73}--\blpage{114}
(\byear{2020})
\doiurl{10.4171/JST/286}
\end{barticle}
\endbibitem

\end{thebibliography}

\end{document}